\documentclass{article}

\usepackage{amsmath}
\usepackage{amsfonts}
\usepackage{amssymb}
\usepackage{bbm}
\numberwithin{equation}{section}

\usepackage{setspace}

\usepackage[a4paper,left=3.18cm, right=3.18cm, top=2.54cm, bottom=2.54cm]{geometry}
\usepackage{enumitem}
\setlist[enumerate]{leftmargin=.5in}
\setlist[itemize]{leftmargin=.5in}

\usepackage{xcolor}
\definecolor{bleudefrance}{rgb}{0.19, 0.55, 0.91}
\definecolor{trueblue}{rgb}{0.0, 0.45, 0.81}
\definecolor{columbiablue}{rgb}{0.61, 0.87, 1.0}
\definecolor{lightgoldenrodyellow}{rgb}{0.98, 0.98, 0.82}

\usepackage{algorithm}
\usepackage{algpseudocode}
\usepackage{float}  

\usepackage{multirow}
\usepackage{graphicx} 
\usepackage{subcaption}
\usepackage{epstopdf}
\usepackage{overpic}

\newlength{\Oldarrayrulewidth}
\newcommand{\Cline}[2]{
  \noalign{\global\setlength{\Oldarrayrulewidth}{\arrayrulewidth}}
  \noalign{\global\setlength{\arrayrulewidth}{#1}}\cline{#2}
  \noalign{\global\setlength{\arrayrulewidth}{\Oldarrayrulewidth}}}

\usepackage{amsthm}
\newtheorem{lemma}{Lemma}
\newtheorem{theorem}{Theorem}

\usepackage[colorlinks=true, allcolors=trueblue]{hyperref}

\usepackage{authblk}
\newcommand*\samethanks[1][\value{footnote}]{\footnotemark[#1]}
\makeatletter  
\def\@fnsymbol#1{\ensuremath{\ifcase#1\or \dagger\or \ddagger\or
   \mathsection\or \mathparagraph\or \|\or **\or \dagger\dagger
   \or \ddagger\ddagger \else\@ctrerr\fi}}
\makeatletter

\usepackage{titlesec}
\definecolor{MSBlue}{rgb}{.204,.353,.541}
\titleformat*{\section}{\Large\bfseries\sffamily\color{black}}
\titleformat*{\subsection}{\large\bfseries\sffamily\color{black}}

\renewenvironment{abstract}{
  \quotation
  \textbf{\textsf{\color{black}{\abstractname.}}} 
  \sffamily
}{\endquotation}

\def\0{\mathbf{0}}
\def\1{\mathbf{1}}

\def\c{\mathbf{c}}
\def\f{\mathbf{f}}
\def\bf{\mathbbm{f}}
\def\g{\mathbf{g}}
\def\h{\mathbf{h}}
\def\p{\mathbf{p}}
\def\r{\mathbf{r}}
\def\s{\mathbf{s}}

\def\v{\mathbf{v}}

\def\x{\mathbf{x}}
\def\y{\mathbf{y}}
\def\z{\mathbf{z}}

\def\L{\mathcal{L}}

\def\R{\mathbb{R}}
\def\A{\mathcal{A}}
\def\E{\mathcal{E}}
\def\F{\mathcal{F}}

\def\M{\mathcal{M}}

\def\V{\mathcal{V}}
\def\I{\mathtt{I}}
\def\B{\mathtt{B}}
\def\C{\mathtt{C}}
\def\X{\mathtt{X}}
\def\Y{\mathtt{Y}}
\def\d{\mathrm{d}}

\def\btheta{\boldsymbol{\theta}}

\def\diag{{\mathrm{diag}}}

\def\trace{\mathrm{trace}}
\newcommand{\argmin}{\operatornamewithlimits{argmin}}

\newcommand{\mean}{\operatornamewithlimits{mean}}
\newcommand{\SD}{\operatornamewithlimits{SD}}

\title{ 
\bfseries\sffamily\color{black}{Energy-Based Distortion-Balancing Parameterization for Open Surfaces}
}
\author{
\sffamily\color{black}
{Shu-Yung Liu\thanks{\footnotesize Department of Mathematics, National Taiwan Normal University, Taipei, Taiwan (\href{mailto:lii227857@gmail.com}{lii227857@gmail.com}, \href{mailto:yue@ntnu.edu.tw}{yue@ntnu.edu.tw})}
~and Mei-Heng Yueh\samethanks[1] }  
}

\date{}

\begin{document}
\begin{sloppypar}

\captionsetup[figure]{labelfont={bf,sf},name={Fig.},labelsep=period} 
\captionsetup[table]{labelfont={bf,sf},name={Table},labelsep=period} 

\maketitle

\begin{abstract}
Surface parameterization is a fundamental concept in fields such as differential geometry and computer graphics. It involves mapping a surface in three-dimensional space onto a two-dimensional parameter space. This process allows for the systematic representation and manipulation of surfaces of complicated shapes by simplifying them into a manageable planar domain. In this paper, we propose a new iterative algorithm for computing the parameterization of simply connected open surfaces that achieves an optimal balance between angle and area distortions. We rigorously prove that the iteration in our algorithm converges globally, and numerical results demonstrate that the resulting mappings are bijective and effectively balance angular and area accuracy across various triangular meshes. Additionally, we present the practical usefulness of the proposed algorithm by applying it to represent surfaces as geometry images.


\end{abstract}

\section{Introduction}
A surface parameterization is a bijective mapping that maps a surface to a canonical domain in $\mathbb R^2$. It plays a crucial role in simplifying various tasks in computer graphics, such as surface resampling, re-meshing, and registration \cite{LaLu14,LuLY14,YoMY14,YuLW17}. Comprehensive overviews of surface parameterization methods and their applications can be found in survey papers \cite{FlHo05,ShPR06,HoLe07}. An ideal parameterization preserves as much of the original surface's geometric properties as possible. Most existing approaches focus on either angle-preserving (conformal) or area-preserving (authalic) parameterizations.

For computing a conformal parameterization, feasible approaches have been established, such as harmonic energy minimization \cite{GuYa08,HuGL14,HuGH14,MuTA08}, Laplacian operator linearization \cite{AnHT99,HaAT00}, angle-based flattening method \cite{ShSt01,ShLM05}, heat diffusion on the double-covered surface \cite{GuYa08}, the heat diffusion by the quasi-implicit Euler method \cite{HuGH14}, the algorithm based on the composition of Cayley transforms and quasi-conformal maps \cite{ChLu15}, the linear algorithm for computing conformal parameterizations \cite{ChLu18}, conformal energy minimization (CEM) \cite{YuLW17,KuLY21} and the constructive algorithm for disk conformal parameterizations \cite{LiYu22}. On the other hand, for computing authalic parameterizations, available numerical methods include stretch-minimizing method \cite{SaSG01,YoBS04}, optimal mass transportation (OMT) method \cite{DoTa10,ZhSG13,SuCQ16}, diffusion-based method \cite{ZoHG11,ChRy18}, the stretch energy minimization (SEM) \cite{YuLW19,Yueh23}, and the authalic energy minimization (AEM) \cite{LiYu24}.

One of the practical applications of parameterization is geometry images 
\cite{GuGH02}. 
A geometry image is a method used in computer graphics and geometric processing to represent the geometry of a 3D surface in a planar image format. This approach begins by parameterizing the surface onto a rectangular planar domain. Once parameterized, geometric information, such as vertex positions, normals, or other attributes, is sampled at regular intervals to form a uniform grid. Each grid point corresponds to a specific location on the surface and stores its geometric properties in 3D space. The sampled data is then encoded into a planar image, where pixel values represent the geometric properties of the surface. For instance, the RGB channels of an image can encode the $x$, $y$, and $z$ coordinates of the surface vertices in 3D. Such an image can reconstruct the original 3D surface by extracting the vertex positions from the image and mapping them back to 3D space. Different parameterizations of the same surface naturally result in different geometry images. Conformal parameterization tends to have a better quality of triangles in the reconstructed surface, while authalic parameterization promotes a more uniform distribution of triangles across the surface.

The distortion-balancing parameterization seeks a trade-off between minimizing angular distortion and area distortion, as it is typically impossible to minimize both simultaneously. Different studies approach this trade-off in various ways. In \cite{NaSZ17}, the trade-off is expressed by the measure $\mu_t$, a convex combination of the initial conformal image area element and the original area element in the optimal mass transportation (OMT) process, with a trade-off value set at $t = 0.5$, meaning an equal balance between the two. In contrast, the trade-off of \cite{ChSh24} is a convex combination of the Beltrami coefficients of conformal and area-preserving mappings, referred to as the balanced Beltrami coefficient $\mu_B$. The corresponding mapping is computed using the Laplace-Beltrami solver with the coefficient of the convex combination determined by the surface reconstruction error. For energy-minimization-based methods, \cite{Yueh23} balances the distortion by minimizing the sum of conformal and stretch energy, referred to as balanced energy. Meanwhile, \cite{YuHL20} formulate to a max-min problem. 

To the best of our knowledge, no existing study has achieved a trade-off that equalizes both angular and area distortion while adjusting this balance for different triangular meshes, likely due to challenges in measuring both distortions. In this paper, we address this issue by defining the trade-off as the equality between conformal and authalic energy, formulating it as a constrained optimization problem in which the trade-off value is determined dynamically during the optimization process. We also propose an algorithm to solve this constrained optimization with guaranteed global convergence. Additionally, we apply our resulting parameterization to geometry images, demonstrating its practical application.

\subsection{Contribution}
The contributions of this paper are three-fold.
\begin{itemize}
    \item[(i)] We formulate the disk-shaped distortion-balancing parameterization as a constrained optimization problem and propose an efficient algorithm to solve this problem with guaranteed global convergence. 
    \item[(ii)] Numerical experiments indicate that our proposed method achieves parameterizations with well-balanced distortions between angle and area across various triangular meshes.
    \item[(iii)] We demonstrate the practical application of distortion-balancing parameterization in representing surfaces in 3-dimensional space as planar images, known as geometry images.
\end{itemize}

\subsection{Notation}
In this paper, we use the following notations.
\begin{itemize}
\item Real-valued vectors are represented using bold letters, such as $\mathbf{f}$, and real-valued matrices are represented using uppercase letters, such as $L$, or blackboard bold lowercase letters, such as $\bf$.
\item Ordered sets of indices are denoted using typewriter letters, such as $\mathtt{I}$ and $\mathtt{B}$.
\item The $i$th entry of the vector $\mathbf{f}$ is denoted by $\mathbf{f}_i$, and the subvector of $\mathbf{f}$, composed of $\mathbf{f}_i$ for $i$ in the set of indices $\mathtt{I}$, is represented as $\mathbf{f}_\mathtt{I}$.
\item The $(i,j)$th entry of the matrix $L$ is represented by $L_{i,j}$, and the submatrix of $L$, composed of entries $L_{i,j}$ for $i$ in the set of indices $\mathtt{I}$ and $j$ in the set of indices $\mathtt{J}$, is denoted by $L_{\mathtt{I},\mathtt{J}}$.
\item The $i$th row of the matrix $\bf$ is represented by $\bf_i$, the $j$th colume of the matrix $\bf$ is represented by $\bf^j$, and the submatrix of $\bf$, composed of the row of the matrix $\bf_i$ for $i$ in the set of indices $\I$, is denoted by $\bf_\I$.
\item The set of real numbers is represented by $\mathbb{R}$.
\item The $k$-simplex with vertices ${v}_0, \ldots, {v}_k$ is denoted by $\left[{v}_0, \ldots, {v}_k\right]$.
\item The area of a $2$-simplex $\left[{v}_0, v_1, {v}_2\right]$ is represented by $|\left[{v}_0, v_1, {v}_2\right]|$, and the area of a simplicial $2$-complex $\M$ is denoted by $|\M|$.
\item The zero and one vectors or matrices of appropriate sizes are denoted by $\mathbf{0}$ and $\mathbf{1}$, respectively.
\end{itemize}

\subsection{Organization of the Paper}
The paper is organized as follows. In Sect. \ref{sec:2}, we introduce conformal energy and authalic energy, defined on simplicial mappings, along with their properties corresponding to conformal and authalic mappings, respectively. Sect. \ref{sec:3} presents the formulation of the distortion-balancing mapping problem as a constrained optimization problem and introduces the associated algorithm. The convergence of the proposed method is discussed in Sect. \ref{sec:4}, followed by numerical experiments in Sect. \ref{sec:5}. In Sect. \ref{sec:6}, we explore the application of geometry images using the distortion-balancing parameterization and present numerical results. Finally, we conclude the paper in Sect. \ref{sec:7}.

\section{Energy for Simplicial Mapping}
\label{sec:2}

A simplicial surface $\M$ embedded in $\R^3$ is composed of $n$ vertices
\begin{subequations} \label{eq:mesh}
\begin{equation}
\mathcal{V}(\M) = \left\{ v_\ell = (v_\ell^1, v_\ell^2, v_\ell^3) \in\mathbb{R}^3 \right\}_{\ell=1}^n,
\end{equation}
$m$ oriented triangular faces 
\begin{equation}
\mathcal{F}(\M) = \left\{ \tau_s = [ {v}_{i_s}, {v}_{j_s}, {v}_{k_s} ] \subset\R^3 \mid {v}_{i_s}, {v}_{j_s}, {v}_{k_s}\in\mathcal{V}(\M) \right\}_{s=1}^m,
\end{equation}
and edges
\begin{equation}
\E(\M) = \left\{ [v_i,v_j] \subset\R^3 \mid [v_i,v_j,v_k]\in\mathcal{F}(\M) \right\},
\end{equation}
where the bracket $[{v}_{i_s}, {v}_{j_s}, {v}_{k_s}]$ denotes a $2$-simplex, i.e., a triangle with vertices being ${v}_{i_s}, {v}_{j_s}, {v}_{k_s}$. 
\end{subequations}

\begin{figure}
\centering
\includegraphics{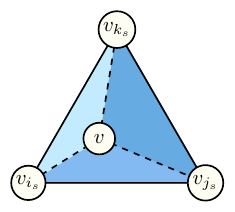}
\caption{An illustration of the areas $|[v,v_{j_s},v_{k_s}]|$, $|[v_{i_s},v,v_{k_s}]|$ and $|[v_{i_s},v_{j_s},v]|$ in barycentric coordinates on a triangular face.}
\label{fig:barycentric}
\end{figure}

A simplicial mapping $f:\M\to\R^2$ is a piecewise affine mapping, which maps to a simplicial complex $f(\M)$, composed of vertices $\V(f(\M))$, edges $\E(f(\M))$, and triangular faces $\F(f(\M))$ as in \eqref{eq:mesh}. 
Specifically, for each $v_i \in \V(\M)$, the vertices of $f(\M)$, 
$$
f(v_i) := \big( f^1(v_i), f^2(v_i) \big)^\top. 
$$
For the point $v \in \tau_s$ for every $\tau_s\in\F(\M)$, the mapping $f(v)$ is determined by vertices $ f(v_{i_s})$, $f(v_{j_s})$, and $f(v_{k_s})$, which can be expressed by barycentric coordinate as follows
$$
f|_{\tau_s}(v) =  \frac{1}{|\tau_s|} \Big(|[v,v_{j_s},v_{k_s}]| \, f(v_{i_s}) + |[v_{i_s},v,v_{k_s}]| \, f(v_{j_s}) + |[v_{i_s},v_{j_s},v]| \, f(v_{k_s}) \Big),
$$
where $|\tau_s|$ denote the area of the triangle $\tau_s$, and the coefficients $|[v,v_{j_s},v_{k_s}]|$, $|[v_{i_s},v,v_{k_s}]|$ and $|[v_{i_s},v_{j_s},v]|$ denote the area of the triangles $[v,v_{j_s},v_{k_s}]$, $[v_{i_s},v,v_{k_s}]$ and $[v_{i_s},v_{j_s},v]$, respectively, illustrated in Fig. \ref{fig:barycentric}.
Thus, the simplicial mapping $f$ can be represented as a matrix
\begin{equation} \label{eq:matrix_f}
\bf \equiv 
\begin{bmatrix}
    f^1(v_1)   &   f^2(v_1)\\
    \vdots     &   \vdots\\
    f^1(v_n)   &   f^2(v_n)
\end{bmatrix} \equiv 
\begin{bmatrix}
    \bf_1 \\
    \vdots\\
    \bf_n
\end{bmatrix} \equiv
\big[ \bf^1, ~\bf^2 \big] 
\in \R^{n \times 2}.
\end{equation}

The conformal and authalic energies, introduced in Sect. \ref{sec:2.1}, are used to measure the mapping's global angle-preserving and area-preserving properties.

\subsection{Conformal and Authalic Energy}
\label{sec:2.1}
Conformal maps (angle-preserving maps) can be computed by minimizing the Dirichlet energy. The discrete Dirichlet energy functional \cite{GuYa08} for a simplicial mapping $f: \mathcal{M} \rightarrow \mathbb R^2$ is defined by
\[
E_D(\bf) = \frac{1}{2} \trace (\bf^\top L_D \bf) = \frac{1}{2} \sum_{s = 1, 2} {\bf^s}^\top L_D \bf^s
\]
where $\bf$ is the matrix representation of the simplicial mapping $f$ as \eqref{eq:matrix_f}, and the cotangent Laplacian matrix $L_D$ is given by
\begin{equation} \label{eq:Ld}
[L_D]_{i, j} = 
\begin{cases}
    -\frac{1}{2} (\cot\theta_{i, j}^k + \cot \theta_{j, i}^\ell)    &\text{if}~ [v_i, v_j] \in \mathcal{E(M)}\\
    -\sum_{\ell \neq i} [L_D]_{i, \ell}     &\text{if}~j=i\\
    0       &\text{otherwise}
\end{cases}
\end{equation}
in which $\theta_{i,j}^k$ is the angle opposite to the edge $[v_i, v_j]$ at the point $v_k$, as illustrated in Fig. \ref{fig:cot}. It is known that $E_D(\bf) \geq |f(\M)|$ and the equality holds if and only if $f$ is a conformal map \cite{Hutc91}, where $|f(\M)|$ is the image area. Once boundary condition $f|_{\partial \M}$ is given, the minimizer of the Dirichlet energy functional $E_D(\bf)$ is a harmonic map. A conformal map is a harmonic map with the appropriate boundary condition such that $E_D(f) = |f(\M)|$. We denote the image area $|f(\M)|$ by the area measurement $\A(\bf)$. The conformal energy of $f$ is thus defined as 
\begin{equation}
E_C(\bf) = E_D(\bf) - \A (\bf), \label{eq:Ec}
\end{equation}
which attains a minimum value of 0

\begin{figure}[]
\centering
\includegraphics{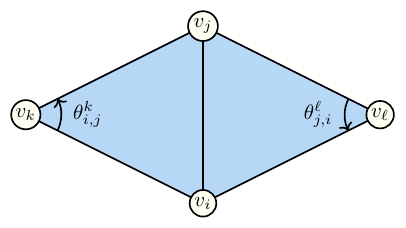}
\caption{An illustration of the angle $\theta_{i,j}^k$ and $\theta_{i,j}^\ell$ defined on the surface $\M$.}
\label{fig:cot}
\end{figure}

Inspired by Dirichlet energy functional, to extend to the authalic maps (area-preserving maps), Yuel et al. \cite{YuLW19} modify the Laplacian matrix from domain to the image, refer as the stretch Laplacian matrix $L_S(\bf)$, defined by 
\begin{equation} \label{eq:Ls}
[L_S(\bf)]_{i, j} = 
\begin{cases}
    -\frac{1}{2} \big(\frac{\cot(\theta_{i, j}^k(\bf))}{\sigma_{\bf([v_i, v_j, v_k])}} + \frac{\cot (\theta_{j, i}^\ell (\bf))}{\sigma_{\bf([v_j, v_i, v_\ell])}} \big)   &\text{if}~ [v_i, v_j] \in \mathcal{E(M)}\\
    -\sum_{\ell \neq i} [L_S(\bf)]_{i, \ell}     &\text{if}~j=i\\
    0       &\text{otherwise}
\end{cases}
\end{equation}
in which $\theta_{i,j}^k (\bf)$ is the angle opposite to the edge $[\bf_i, \bf_j]$ at the point $\bf_k$, and $\sigma_{\bf}([v_i, v_j, v_k])$ is the stretch factor of $f$ on the triangular face $[v_i, v_j, v_k]$ by
$$
\sigma_{\bf}([v_i, v_j, v_k]) = \frac{|[v_i, v_j, v_k]|}{|[\bf_i, \bf_j, \bf_k]|}.
$$
With this modification, the stretch energy is thus defined as 
\[
E_S(\bf) = \frac{1}{2} \trace (\bf^\top L_S(\bf) \bf) = \frac{1}{2} \sum_{s = 1, 2} {\bf^s}^\top L_S(\bf) \bf^s.
\]
Unexpectedly, the stretch energy exhibits properties similar to those of the Dirichlet energy. It has been shown that $E_S(\bf) \geq |f(\M)|$ and the equality holds if and only if $f$ is an authalic map under the assumption that the total area of the surface remains unchanged \cite{Yueh23}. i.e. $|\M| = |f(\M)|$, where $|\M|$ is the area of surface $\M$. The normalization of the total area of the surface results in the absence of the assumption, which inspired Liu et al \cite{LiYu24} to define the authalic energy as  
\begin{equation}
E_A(\bf) = \frac{|\M|}{\A(\bf)} E_S(\bf) - \A(\bf),
\label{eq:Ea}
\end{equation}
which attains a minimum value of 0, with the minimizer being an authalic map without the constraint of the total area.

\section{Augmented Lagrangian Method for Distortion-Balancing Parameterization}
\label{sec:3}
To balance the conformal and authalic while minimizing them, we formulate the problem to the constrained optimization as 
\begin{equation}
    \min_f E_C(\bf) ~~\text{subject to}~ E_A(\bf) = E_C(\bf).
\label{eq:obj_fun}
\end{equation}
The Lagrange function of the objective problem is 
\begin{align}
    \L(\bf, \lambda) &= E_C(\bf) + \lambda (E_A(\bf) - E_C(\bf)). \label{eq:Lag}
\end{align}
We can regard the constrained minimization problem \eqref{eq:obj_fun} by minimizing the 
 discontinuous function $F(\bf)$, which is defined by 
\begin{align*}
    F(\bf)  &= \max_{\lambda} \L (\bf, \lambda) \\ 
            &= \max_\lambda \{ E_C(\bf) + \lambda (E_A(\bf) - E_C(\bf)) \}\\
            &=
            \begin{cases}
             E_C(\bf)  &  \text{if } E_A(\bf) = E_C(\bf)\\
             \infty &   \text{otherwise}
            \end{cases}
\end{align*}
and thus
\[
\min_\bf F(\bf) = \min_\bf \max_{\lambda} \L(\bf, \lambda) = \min_\bf E_C(\bf) ~~\text{subject to}~ E_A(\bf) = E_C(\bf).
\]
Suppose the strong duality holds, then we obtain
\begin{equation}
\min_\bf \max_{\lambda} \L(\bf, \lambda) = \max_{\lambda} \min_\bf \L(\bf, \lambda). \label{eq:max min}
\end{equation}

The augmented Lagrangian method \cite{CoGT91} is an iterative method to solve $\max_{\lambda} \min_\bf \L(\bf, \lambda)$. The augmented Lagrangian is defined with an additional penalty term, given by
\begin{equation}
    \L_A(\bf, \lambda, \rho) = E_C(\bf) + \lambda (E_A(\bf) - E_C(\bf)) + \frac{\rho}{2} (E_A(\bf) - E_C(\bf))^2, \label{eq:AugLag}
\end{equation}
where $\rho > 0$ is a penalty parameter. The penalty plays a crucial role. As long as the estimate of $\lambda$ remains reasonably accurate, driving the penalty parameter $\rho^{(k)} \uparrow \infty$ leads to convergence, as in the quadratic penalty method \cite{GoNi89}. The common approach for updating $\lambda$ is the first-order estimate with the step size $\rho$. In $k$-th iteration,
\[
\lambda^{(k+1)} \leftarrow \lambda^{(k)}  + \rho^{(k)} \big(E_A(\bf^{(k+1)}) - E_C(\bf^{(k+1)}) \big).
\]
As a result, the augmented Lagrangian method for solving disk-shaped mapping $f$ can be summarized by the following key steps. 
\begin{subequations} \label{eq:pseudo_alg}
\begin{align}
     \bf^{(k+1)} &\leftarrow \argmin_\bf \Big\{ \L_A(\bf, \lambda^{(k)}, \rho^{(k)}) ~\Big|~ \bf_i \in \mathbb D ~\mbox{for}~ i = 1,\cdots, n \Big\}\label{eq:f update}\\
    \lambda^{(k+1)} &\leftarrow \lambda^{(k)} + \rho^{(k)}\, \big(E_A(\bf^{(k+1)}) - E_C(\bf^{(k+1)})\big) \label{eq:lambda update}\\
    \rho^{(k+1)} &\leftarrow \tau \rho^{(k)}, \label{eq:rho update}
\end{align}
\end{subequations}
where $\mathbb D$ is a unit disk and $\bf_i$ is defined as \eqref{eq:matrix_f}, providing the amplifier factor $\tau > 1$. 

The disk-shaped constraint in \eqref{eq:f update} can be implemented by representing the boundary points in polar coordinates with a fixed radius of 1. The associated energy and gradient details are presented in Sect. \ref{sec:2.1}. Using this formulation, the minimization in \eqref{eq:f update} can be efficiently solved using the preconditioned nonlinear conjugate gradient (CG) method described in Sect. \ref{sec:3.2}. To ensure that the first-order estimate of $\lambda$ behaves appropriately, it is essential to provide proper criteria for the update procedures \eqref{eq:lambda update} and \eqref{eq:rho update}, ensuring the global convergence \eqref{eq:f update}-\eqref{eq:rho update}, which is explained in Sect. \ref{sec:3.3}. Furthermore, providing a square-shaped version is valuable for practical applications discussed in Sect. \ref{sec:3.4}.

\subsection{Reformulation in Polar Coordinate}
\label{sec:3.1}
To ensure the disk-shaped mapping $f$, we first introduce the index sets of the interior and boundary vertices, 
\begin{equation} \label{eq:BI}
\B = \{ b \mid v_b \in \partial\M \} ~ \text{ and } ~ \I = \{1, \ldots, n\}\backslash\B,
\end{equation}
with the number $n = n_\I + n_\B \equiv \#(\I) + \#(\B)$. Since the gradient of Dirichlet energy and stretch energy \cite{Yueh23} is 
\[
\nabla_{\bf^s} E_D(\bf) = L_D \bf^s, ~~~ \nabla_{\bf^s} E_S(\bf) = 2 L_S(\bf) \bf^s,
\]
for $s = 1, 2$, then, by a certain reordering of the indices of vertices, the matrix $L_D$, $L_S(\bf)$, and $\bf$ can be written as 
\[
L_D = 
\begin{bmatrix}
    [L_D]_{\I, \I}  &  [L_D]_{\I, \B}\\
    [L_D]_{\B, \I}  &  [L_D]_{\B, \B}
\end{bmatrix}, ~~~ L_S(\bf) = 
\begin{bmatrix}
    [L_S(\bf)]_{\I, \I}  &  [L_S(\bf)]_{\I, \B}\\
    [L_S(\bf)]_{\B, \I}  &  [L_S(\bf)]_{\B, \B}
\end{bmatrix}, ~\mbox{and}~ \bf = 
\begin{bmatrix}
    \bf_\I\\
    \bf_\B
\end{bmatrix},
\]
where $\bf_\I \in \mathbb R^{n_\I \times 2}$ and $\bf_\B \in \mathbb R^{n_\B \times 2}$. Therefore, the gradient of Dirichlet energy and stretch energy with respect to the interior points $\bf_\I$ and boundary points $\bf_\B$ are
\begin{align}
    \nabla_{\bf_\I} E_D(\bf) &= [L_D]_{\I, \I}  \bf_\I + [L_D]_{\I, \B} \bf_\B \in \mathbb R^{n_\I \times 2} \label{eq:dED}\\
    \nabla_{\bf_\B} E_D(\bf) &= [L_D]_{\B, \I}  \bf_\I + [L_D]_{\B, \B} \bf_\B \in \mathbb R^{n_\B \times 2}. \nonumber
\end{align}
and
\begin{align}
    \nabla_{\bf_\I} E_S(\bf) &= 2\big( [L_S(\bf)]_{\I, \I}  \bf_\I + [L_S(\bf)]_{\I, \B} \bf_\B \big) \in \mathbb R^{n_\I \times 2} \label{eq:dEs}\\
    \nabla_{\bf_\B} E_S(\bf) &= 2\big( [L_S(\bf)]_{\B, \I}  \bf_\I + [L_S(\bf)]_{\B, \B} \bf_\B \big) \in \mathbb R^{n_\B \times 2}. \nonumber
\end{align}

Next, to remain boundary points $\bf_\B$ on the unit circle, $\bf_\B \in \mathbb R^{n_\B \times 2}$ are represented by the polar coordinate with radius $1$. i.e.
\[
\bf_\B = 
\begin{bmatrix}
    \bf_\B^1     &   \bf_\B^2
\end{bmatrix} = 
\begin{bmatrix}
    \cos \theta_1   &   \sin \theta_1\\
    \vdots          &   \vdots\\
    \cos \theta_{n_\B}   &   \sin \theta_{n_\B}
\end{bmatrix} =
\begin{bmatrix}
    \cos\btheta     &   \sin \btheta
\end{bmatrix} \equiv
\begin{bmatrix}
    \c  &  \s
\end{bmatrix}
\]
in which $\btheta = (\theta_1, \cdots, \theta_{n_\B})^\top$. Therefore, the variable of disk-shaped simplicial mapping can be represented as a vector
\begin{equation}
\f = 
\begin{bmatrix}
    \bf_\I^1\\
    \bf_\I^2\\
    \btheta
\end{bmatrix} \in \mathbb R^{2 n_\I + n_\B}.
\label{eq:hat f}
\end{equation}

The image area $\A(\bf)$ is a polygon that approximated to the unit disk can be formulated by
\begin{align}
\A(\bf) &= \frac{1}{2} \sum_{i=1}^{n_\B} \sin (\theta_{i+1} - \theta_i) \quad\quad (\theta_{n_\B+1}:=\theta_1) \nonumber\\
&= \frac{1}{2}
\begin{bmatrix}
\cos \theta_1\\
\cos \theta_2\\
\vdots\\
\cos \theta_{n_\B}
\end{bmatrix}^{\top}
\begin{bmatrix}
0  &  1     &       &  -1\\
-1 &  0 &\ddots &    \\
   & \ddots &\ddots &  1 \\
1  &        &-1     &0
\end{bmatrix}
\begin{bmatrix}
\sin \theta_1\\
\sin \theta_2\\
\vdots\\
\sin \theta_{n_\B}
\end{bmatrix}\\ 
&\equiv \frac{1}{2} \c^{\top} D\, \s  \equiv \A(\btheta) \label{eq:PolarArea}
\end{align} 
and its gradient 
\begin{equation}
\nabla \A(\btheta) = - \frac{1}{2} \big( \diag (\c) D \c + \diag (\s) D \s \big).
\end{equation}
Detailed derivations can be found in \cite{KuLY21}.

With the total area formulation \eqref{eq:PolarArea}, we regard $\bf$ as $\bf(\f)$. Thus, the conformal energy $E_C(\bf)$ in \eqref{eq:Ec}, and authalic energy $E_A(\bf)$ in \eqref{eq:Ea}, can be expressed by 
\begin{equation} \label{eq:polar_Ec}
E_C(\f) = \frac{1}{2} \sum_{s = 1, 2} {\bf^s}^\top L_D  \bf^s -  \frac{1}{2} \c^{\top} D\, \s
\end{equation}
and 
\begin{equation} \label{eq:polar_Ea}
E_A(\f) = \frac{2 |\M|}{\c^{\top} D\, \s} E_S(\f) - \frac{1}{2} \c^{\top} D\, \s, ~~~ E_S(\f) = \frac{1}{2} \sum_{s = 1, 2}  {\bf^s}^\top L_S(\f)  \bf^s,
\end{equation}
respectively. Further, with the gradient of $E_D$ and $E_S$ with respect to $\bf_\I$ and $\bf_\B$ in \eqref{eq:dED} and \eqref{eq:dEs}, respectively, the gradient of $E_C$ and $E_A$ \cite{LiYu24} with respect to $\f$ can be formulated by 
\begin{subequations}    \label{eq:dEc}
\begin{align}
    \nabla_{\bf_\I^1} E_C(\f) &= [L_D]_{\I, \I}  \bf_\I^1 + [L_D]_{\I, \B} \c \\
    \nabla_{\bf_\I^2} E_C(\f) &= [L_D]_{\I, \I}  \bf_\I^2 + [L_D]_{\I, \B}  \s \\
    \nabla_{\btheta} E_C(\f) &= \diag (\c) \big( [L_D]_{\B, \I}  \bf_\I^2 + [L_D]_{\B, \B} \s \big) -\diag (\s) \big( [L_D]_{\B, \I}  \bf_\I^1 + [L_D]_{\B, \B} \c \big) \nonumber\\
                    &+\frac{1}{2} (\diag (\c) D \c + \diag(\s) D \s),
\end{align}
\end{subequations}
and
\begin{subequations} \label{eq:dEa}
\begin{align}
    \nabla_{\bf_{\I}^1} E_A(\f)  &= \Big( \frac{ 2 |\M|}{\c^{\top} D \s} \Big) \nabla_{\bf_{\I}^1} E_S(\f) \nonumber\\
        &= \Big( \frac{4 |\M|}{\c^{\top} D \s} \Big) \,( [L_S(\f)]_{\I, \I} \bf_{\I}^2 + [L_S(\f)]_{\I, \B} \c)\\
    \nabla_{\bf_{\I}^2} E_A(\f)  &= \Big( \frac{ 2 |\M|}{\c^{\top} D \s} \Big) \nabla_{\bf_{\I}^2} E_S(\f) \nonumber\\
        &= \Big( \frac{4|\M|}{\c^{\top} D \s} \Big) \,( [L_S(\f)]_{\I, \I} \bf_{\I}^2 + [L_S(\f)]_{\I, \B} \s)\\
    \nabla_{\btheta} E_A(\f) &= \nabla_{\btheta} \Big( \frac{2 |\M|}{\c^{\top} D \s} \Big)  E_S(\f) + \Big( \frac{2 |\M|}{\c^{\top} D \s} \Big) \nabla_{\btheta} E_S(\f) - \nabla_{\btheta} \A(\btheta)  \nonumber\\
    &= \frac{4 |\M|}{\c^{\top} D \s}\, \big( \diag(\c)( [L_S(\f)]_{\B, \I} \bf^2_{\I} + [L_S(\f)]_{\B, \B} \s) - \diag(\s)( [L_S(\f)]_{\B, \I} \bf^1_{\I} + [L_S(\f)]_{\B, \B} \c) \big) \nonumber\\
        &+ \Big(\frac{1}{2} + \frac{|\M| E_S(\f))}{(\c^{\top} D \s)^2} \Big) (\diag (\c) D \c + \diag(\s) D \s).
\end{align}
\end{subequations}

\subsection{Preconditioned Nonlinear CG Method}
\label{sec:3.2}
The minimization \eqref{eq:f update} can be solved efficiently by the preconditioned nonlinear CG method \cite{LiYu24}, which has the theoretically guaranteed convergence. Specifically, the disk-shaped simplicial mapping $f$ can be represented with polar coordinate $\f \in \mathbb R^{2 n_\I + n_\B}$ as in \eqref{eq:hat f}. With fix the parameter $\lambda^{(k)}$ and $\rho^{(k)}$, the objective functional is then denoted by
\begin{equation} \label{eq:PCG_obj}
\L_{(\lambda^{(k)}, \rho^{(k)})}(\f) := \L_A (\bf, \lambda^{(k)}, \rho^{(k)}).
\end{equation}
The gradient of $\L_{(\lambda^{(k)}, \rho^{(k)})}$ with respect to $\f^{(\ell)}$ in $\ell$-th iteration, denoted as $\g^{(\ell)}$, defined by
\begin{equation} \label{eq:grad L}
\nabla \L_{(\lambda^{(k)}, \rho^{(k)})}(\f) = \nabla E_C(\f) + \big( \lambda^{(k)} + \rho^{(k)} (E_A(\f) - E_C(\f)) \big) \nabla \big(E_A(\f) - E_C(\f) \big), 
\end{equation}
where $\nabla E_C$ and $\nabla E_A$ are computed by \eqref{eq:dEc} and \eqref{eq:dEa}.

The initial mapping $\f^{(0)}$ is selected by the outcome mapping of fixed-point method \cite{YuHL20} with 5 iterations, presented in Appendix A. The updating of $\f^{(\ell)}$ in $\ell$-th iteration is by the iterative method with the following procedure,
\[
\f^{(\ell+1)} = \f^{(\ell)} + \alpha_\ell \p^{(\ell)}
\]
produced by the suitable step length $\alpha_\ell$ and the direction 
\[
 \p^{(\ell)} = -M^{-1} \g^{(\ell)} + \beta_\ell \p_\I^{(\ell-1)} ~~\text{with}~~ \beta_\ell = \frac{ {\g^{(\ell)}}^\top M^{-1}  \g^{(\ell)} }{  {\g^{(\ell-1)}}^\top M^{-1}  \g^{(\ell-1)} },
\]
and the preconditioner $M$, where $\p^{(0)} = - M^{-1} \g^{(0)}$ when $\ell = 0$. The preconditioner $M$ has to be the symmetric positive definite matrix by choosing
\begin{equation}
M = 
\begin{bmatrix}
    I_2 \otimes [L_{\lambda^{(k)}}(\f^{(0)})]_{\I, \I} & \\
    &  [L_{\lambda^{(k)}}(\f^{(0)})]_{\B, \B}
\end{bmatrix},
\label{eq:preconditioner}
\end{equation}
where the matrix $L_{\lambda^{(k)}}(f)$ is defined by the linear combination of $L_D$ and $L_S(\f)$,
\begin{equation}
L_{\lambda^{(k)}} (\f) = (1 - \lambda^{(k)}) L_D + \frac{2 |\M| \lambda^{(k)}}{\A(\btheta)} L_S(\f).
\label{eq:Lrho}
\end{equation}
It is worth noting that $L_{\lambda^{(k)}} (\f^{(0)})$ is a Laplacian matrix if both coefficients of $L_D$ and $L_S(\f^{0})$ are positive, which is equivalent to $\lambda^{(k)} \in [0, 1]$. This can be achieved with proper updating of $\lambda^{(k)}$ because the Lagrange multiplier $\lambda^*$ is in $[0, 1]$ by the following lemma.
\begin{lemma} \label{lma:lambda}
    The Lagrange multiplier $\lambda^* = \max_\lambda \min_\f \L(\f, \lambda)$ is in $[0, 1]$.
\end{lemma}
\begin{proof}
    Suppose $\lambda^* < 0$. For a conformal mapping $\f_C$, since $E_A(\f_C) \geq 0$ and $E_C(\f_C) = 0$, we have
    \begin{align*}
    \min_\f \L(\f, \lambda^*) \leq \L(\f_C, \lambda^*) &= (1-\lambda^*) E_C(\f_C) + \lambda^* E_A(\f_C) \\
    &= \lambda^* E_A(\f_C) < 0 = E_C(\f_C) = \min_\f \L(\f, 0),
    \end{align*}
    which is contradict to $\lambda^* = \max_\lambda \min_\f \L(\f, \lambda)$.
    On the other hand, suppose $\lambda^* > 1$, for a authalic mapping $\f_A$, since $E_C(\f_A) \geq 0$ amd $E_A(\f_A) = 0$, we have
    \begin{align*}
    \min_\f \L(\f, \lambda^*) \leq \L(\f_A, \lambda^*) &= (1-\lambda^*) E_C(\f_A) + \lambda^* E_A(\f_A) \\
    &= (1-\lambda^*) E_C(\f_A) < 0 = E_A(\f_A) = \min_\f \L(\f, 1),
    \end{align*}
    which is contradict to $\lambda^* = \max_\lambda \min_\f \L(\f, \lambda)$.
    As a result, we conclude $\lambda^* \in [0, 1]$.
\end{proof}

An ideal step length $\alpha^*$ satisfy $\alpha^* = \argmin_\alpha \L_{(\lambda^{(k)}, \rho^{(k)})}(\f^{(\ell)} + \alpha \p^{(\ell)})$. However, the complexity of $\L_A$ makes it hard to compute $\alpha^*$ directly. Thus, we form an interpolating quadratic polynomial with available information on energy and gradient and approximate $\alpha^*$ by the minimizer of the polynomial. Specifically, $\alpha$ satisfies
\[
\alpha = \argmin \varphi(\alpha), ~~~ \varphi ( \alpha) = a \alpha^2 + b \alpha + c,
\]
where $\varphi(\alpha)$ is the quadratic function interpolating $\L_{(\lambda^{(k)}, \rho^{(k)})} ({\f}^{(\ell)} + \alpha \p^{(\ell)})$ by the information
\begin{subequations} \label{eq:Step}
\begin{equation} \label{eq:Phi_cond}
\begin{cases}
\varphi(0) =  \L_{(\lambda^{(k)}, \rho^{(k)})} ({\f}^{(\ell)}), \\
\tfrac{\d}{\d\alpha} \varphi(0) = \tfrac{\d}{\d\alpha} \L_{(\lambda^{(k)}, \rho^{(k)})} ({\f}^{(\ell)} + \alpha \p^{(\ell)})\Big|_{\alpha = 0} 
= {{\p}^{(\ell)}}^\top \nabla_{\f} \L_{(\lambda^{(k)}, \rho^{(k)})} ({\f}^{(\ell)}), \\
\varphi(\alpha_{\ell-1}) = \L_{(\lambda^{(k)}, \rho^{(k)})} ({\f}^{(\ell)} + \alpha^{(\ell-1)} \p^{(\ell)})
\end{cases}
\end{equation}
Therefore, we select $\alpha_\ell$ with the approximated step length satisfying $\tfrac{\d}{\d\alpha} \varphi(\alpha) = 0$, which can be explicitly formulated as
\begin{equation}
\alpha_\ell = - \frac{b}{2a},
\end{equation}
where the coefficient $a$, and $b$, solved by \eqref{eq:Phi_cond}, are
\begin{equation}
\begin{cases}
    b = \tfrac{\d}{\d\alpha} \varphi(0),\\
    a = \dfrac{\varphi(\alpha_{\ell-1}) - \varphi(0) - \alpha_{\ell-1} \tfrac{\d}{\d\alpha} \varphi(0) }{\alpha_{\ell-1}^2}.
\end{cases}
\end{equation}
\end{subequations}

Numerically, if the quadratic interpolation step length does not result in a sufficient decrease in energy, the obtained step length can be used as an initial guess. We then perform quadratic interpolation again, repeating the process until the energy decreases adequately.

In practice, for each block $M_i$ of the preconditioner $M$, $i=1,2$, we compute the preordered Cholesky decomposition
\begin{equation} \label{eq:Chol}
U_i^\top U_i = P_i^\top M_i P_i,
\end{equation}
where $U_i$ is an upper triangular matrix and $P_i$ is the approximate minimum degree permutation matrix \cite{AmDD04}. Then, linear systems of the form $M_i\x = \r$ can be efficiently solved by solving lower and upper triangular systems 
\begin{subequations} \label{eq:LS}
\begin{equation} \label{eq:LS1}
U_i^\top \y = P_i^\top \r \,\text{ and }\, U_i\z = \y.
\end{equation}
Ultimately, the solution to the system is given by 
\begin{equation} \label{eq:LS2}
\x = P_i\z. 
\end{equation}
\end{subequations}

As a result, the procedure of minimization \eqref{eq:f update} can be summarized by the Algorithm \ref{alg:PCG}.

\begin{algorithm}[htbp]
\caption{Preconditioned nonlinear CG method for solving \eqref{eq:f update}}
\label{alg:PCG}
\begin{algorithmic}[1]
\Require An initial map $\bf$, multiplier $\lambda$, and penalty parameter $\rho$.
\Ensure A minimizer mapping $\bf^*$ in \eqref{eq:f update}.
\State Let $\I$ and $\B$ defined in \eqref{eq:BI}.
\State Construct $L_D$ and $L_S$ by \eqref{eq:Ld} and \eqref{eq:Ls}. 
\State Set $\btheta = \text{atan2} (\bf_\B^2, \bf_\B^1)$.
\State Construct the polar form $\f$ of $\bf$ as defined by \eqref{eq:hat f}.
\State Compute gradient $\g$ \eqref{eq:grad L} with respect to $\f$ by \eqref{eq:dEc} and \eqref{eq:dEa}.
\State Compute the preconditioner $M$ by \eqref{eq:preconditioner} and \eqref{eq:Lrho}.
\State Compute the Cholesky decomposition $U_i^\top U_i = P_i^\top M_i P_i$, as in \eqref{eq:Chol}, $i=1,2$.
\State Solve $M\h = \g$ by \eqref{eq:LS}.
\State Let $\p = -\h$.
\State Set $\alpha = 0.1$.
\While {not converge}
    \State Update $\alpha$ as \eqref{eq:Step}.
    \State Update $\f \leftarrow \f + \alpha \p$.
    \State Update $L_S$ by \eqref{eq:Ls}.
    \State Let $\gamma = \h^\top \g$.
    \State Update the gradient $\g$ \eqref{eq:grad L} by \eqref{eq:dEc} and \eqref{eq:dEa}.
    \State Solve $M\h = \g$ by \eqref{eq:LS}.
    \State Update $\beta = (\h^{\top} \g) / \gamma$.
    \State Update $\p \gets -\h + \beta \p$.
\EndWhile
\State Obtain $\bf^*$ with the polar form $\f$.
\end{algorithmic}
\end{algorithm}

\subsection{Algorithm of Augmented Lagrangian Method}
\label{sec:3.3}
The updating of $\lambda^{(k)}$, given by \eqref{eq:lambda update}, can be viewed as gradient ascent method with step length $\rho^{(k)}$. Accompanied by amplifying of $\rho^{(k)}$ in \eqref{eq:rho update} with the factor $\tau > 1$, $\lambda^{(k)}$ has the risk of bad behavior. To prevent that, it is essential to set $\rho^{(0)}$ small initially and tightly control the residual $r(\f^{(k)}) = E_A(\f^{(k)}) - E_C(\f^{(k)})$ to well estimate $\lambda^{(k+1)}$ for convergence \cite{CoGT91}. In other words, after updating $\f^{(k)}$ in \eqref{eq:f update} with the criteria
\begin{equation} \label{eq:omega}
\| \nabla_{\f} \L_{(\lambda^{(k)}, \rho^{(k)})} ({\f}^{(\ell)}) \| \leq \omega^{(k)} ~~ \mbox{for some} ~ \omega^{(k)},
\end{equation} 
we update $\lambda^{(k)}$ only if $|r(\f^{(k)})| \leq \eta^{(k)}$ for some criteria $\eta^{(k)}$.
Otherwise, we increase $\rho^{(k+1)} = \tau \rho^{(k)}$ and remain $\lambda^{(k+1)} = \lambda^{(k)}$.

Specifically, the criteria \cite{CoGT91} is set as follow, if $|r(\f^{(k)})| \leq \eta^{(k)}$, we set 
\begin{equation} \label{eq:lambda up}
\begin{cases}
\lambda^{(k+1)} &\leftarrow \lambda^{(k)} + \rho^{(k)}\, r( \f^{(k)}), \\
\rho^{(k+1)}    &\leftarrow \rho^{(k)},\\
 u^{(k+1)}      &\leftarrow \min( 1/\rho^{(k+1)}, \gamma), \\
\omega^{(k+1)}  &\leftarrow \omega^{(k)} (u^{(k+1)})^{t_\omega},\\
\eta^{(k+1)}    &\leftarrow \eta^{(k)} (u^{(k+1)})^{t_\eta}.,
\end{cases}
\end{equation}  
Otherwise, we set
\begin{equation} \label{eq:rho up}
\begin{cases}
\lambda^{(k+1)} &\leftarrow \lambda^{(k)}, \\
\rho^{(k+1)}    &\leftarrow \tau \rho^{(k)},\\
 u^{(k+1)}      &\leftarrow \min( 1/\rho^{(k+1)}, \gamma), \\
\omega^{(k+1)}  &\leftarrow \omega^{(0)} (u^{(k+1)})^{v_\omega},\\
\eta^{(k+1)}    &\leftarrow \eta^{(0)} (u^{(k+1)})^{v_\eta}.
\end{cases}
\end{equation}  
Numerically, we suggest parameters are set by
$\tau = 5$, $t_\omega = v_\omega = 1$, $t_\eta = 0.9$, $v_\eta = 0.5$, $\eta^{(0)} = 0.01$, and $\gamma = \omega^{(0)} = 0.1$.

Moreover, to remain $\lambda^{(k)} \in [0, 1]$ for ensuring the preconditioner $M$ being positive definite in \eqref{eq:preconditioner}, we require tighter criteria that 
\begin{equation} \label{eq:tight lambda}
|r(\f^{(k)})| \leq \min \Big( \eta^{(k)},~ \tfrac{1-\lambda^{(k)} }{\rho^{(k)}},~ \tfrac{\lambda^{(k)} }{\rho^{(k)}} \Big)
\end{equation}  
for the procedure \eqref{eq:lambda up}.

Hence, we summarize the procedure of augmented Lagrangian method \eqref{eq:pseudo_alg} in Algorithm \ref{alg:AugLag}.

\begin{algorithm}[]
\caption{Augmented Lagrangian method for distortion-balancing parameterization}
\label{alg:AugLag}
\begin{algorithmic}[1]
\Require A simply connected open mesh $\mathcal{M}$. 
\Ensure A disk-shaped distortion-balancing map $\bf^*$ in matrix form \eqref{eq:matrix_f}.
\State Let $\I$ and $\B$ be defined in \eqref{eq:BI}.
\State Compute initial mapping $\bf$ by Algorithm \ref{alg:BEM} with $\lambda = 0.4$.
\State Define $L_D$ and $L_S$ as in \eqref{eq:Ld} and \eqref{eq:Ls}.
\State Set $\omega^* = \sqrt{(\# \I + \# \B)} 10^{-4}$.
\State Set $\eta^* = 10^{-5}$.
\State Set $\tau = 5$.
\State Set $\rho = 0.1$.
\State Set $\omega = 0.01$.
\State Set $\eta = 0.01$.
\State Compute polar form $\f$ of $\bf$ as in \eqref{eq:hat f}.
\For{$k = 0, 1, 2, \cdots$}
    \State Find $\f$ satisfying \eqref{eq:omega} by Algorithm \ref{alg:PCG}.
    \State Compute $r = E_C - E_S$ by \eqref{eq:polar_Ec} and \eqref{eq:polar_Ea}.
    \If{ $\f$ satisfying \eqref{eq:omega} with $\omega^*$ and $|r| < \eta^*$}
        \State Stop
    \EndIf
    \If{$|r| \leq \min \big( \eta,~(1-\lambda)/\rho,~\lambda/\rho\big)$}
        \State Set $\lambda = \lambda + \rho \,r$.
        \State Set $u = \min(1/ \rho, ~0.1)$.
        \State Set $\omega = \omega \, u$.
        \State Set $\eta = \eta \, u^{0.9}$.
    \Else
        \State Set $\rho = \tau \rho$.
        \State Set $u = \min(1/ \rho, ~0.1)$.
        \State Set $\omega = 0.1 \, u$.
        \State Set $\eta = 0.01 \, u^{0.5}$.
    \EndIf
    \State $k = k + 1$.
\EndFor
\State Obtain $\bf^*$ with the polar form $\f$.
\end{algorithmic}
\end{algorithm}

\subsection{Modification for Square-shaped}
\label{sec:3.4}
From an application perspective, it is valuable to consider square-shaped parameterization. We use the approach that follows the method in \cite{Yueh23}, where the boundary is constrained by fixing four corner points and allowing the remaining boundary points to slide along the square's edges. The square-shaped parameterization is applied to geometry images in Sect. \ref{sec:6}.

Specifically, we first sort the boundary indices $\B = \{\B(1), \B(2), \cdots, \B(n_\B) \}$ in counterclockwise order. Next, we select the corner indices $\C = \{\C_1, \C_2, \C_3, \C_4 \}$ from $\B$, and maps the corresponding boundary vertices $\v_{\B(\C_1)}$, $\v_{\B(\C_2)}$, $\v_{\B(\C_3)}$, and $\v_{\B(\C_4)}$ to the square’s corners at $(0, 0)$, $(1, 0)$, $(1, 1)$, and $(0, 1)$, respectively. Since the boundary indices $\B$ form a periodic sequence, we set $\C_1 = 1$ and preserve the counterclockwise order. The boundary indices $\B$ are then divided into four segments,
\begin{align*}
    \Y_0 &= \{ i \in \B \,|\, ~ 1  \leq i \leq \C_2 \} \\
    \X_1 &= \{ i \in \B \,|\, \C_2 \leq i \leq \C_3 \} \\
    \Y_1 &= \{ i \in \B \,|\, \C_3 \leq i \leq \C_4 \} \\
    \X_0 &= \{ i \in \B \,|\, \C_4 \leq i \leq n_{\B}\} \cup \{1\}.
\end{align*}
The square-shaped boundary constraint is then formulated as follows,
\begin{equation*}
    \bf^2_{\B(\Y_0)} = \0, ~~ \bf^1_{\B(\X_1)} = \1, ~~ \bf^2_{\B(\Y_1)} = \1, ~~\mbox{and}~~ \bf^1_{\B(\X_0)} = \0.
\end{equation*}
The remaining indices of $\bf^1$ and $\bf^2$ correspond to the interior points, which are defined by
\[
\I_1 = \I \cup \B(\Y_0) \cup \B(\Y_1), ~~\mbox{and}~~ \I_2 = \I \cup \B(\X_0) \cup \B(\X_1),
\] 
respectively. As a result, the square-shaped distortion-balancing mapping can be computed using Algorithm \ref{alg:AugLag} with variables $\bf^1_{\I_1}$ and $\bf^2_{\I_2}$ without polar coordinate.

\section{Convergence Analysis}
\label{sec:4}
In this section, we establish the global convergence of Algorithm \ref{alg:AugLag} for solving the objective function \eqref{eq:obj_fun} to obtain the disk-shaped distortion-balancing parameterization, with inner iterations performed using Algorithm \ref{alg:PCG}.

For the inner iterations \eqref{eq:f update}, utilizing the preconditioned nonlinear CG method, Algorithm \ref{alg:PCG}, convergence is theoretically guaranteed, as stated in the following theorem.

\begin{theorem}
Given $\lambda^{(k)}$ and $\rho^{(k)}$, the preconditioned nonlinear CG method, Algorithm \ref{alg:PCG}, converges globally under the assumption that step lengths satisfy the strong Wolfe conditions,
\begin{subequations} \label{eq:Wolfe}
\begin{align}
    &\L_{(\lambda^{(k)}, \rho^{(k)})} ({\f}^{(\ell)} + \alpha_\ell \p^{(\ell)})  \leq \L_{(\lambda^{(k)}, \rho^{(k)})}({\f}^{(\ell)}) + c_1 \alpha_\ell \nabla \L_{(\lambda^{(k)}, \rho^{(k)})}({\f}^{(\ell)})^{\top} \p^{(\ell)}, \label{eq:Wolfe1} \\
    & \big|\nabla \L_{(\lambda^{(k)}, \rho^{(k)})}({\f}^{(\ell)} + \alpha_\ell \p^{(\ell)})^{\top} \p^{(\ell)} \big|  \leq c_2  \big| \nabla \L_{(\lambda^{(k)}, \rho^{(k)})}({\f}^{(\ell)})^{\top} \p^{(\ell)} \big|, \label{eq:Wolfe2}
\end{align}
\end{subequations}
with $0 < c_1 < c_2 < \frac{1}{2}$.    
\end{theorem}
\begin{proof}
If the preconditioner $M$ is a symmetric positive definite matrix, the convergence of the preconditioned nonlinear CG method is guaranteed under the assumptions of Lipschitz continuity of both the gradient and energy, a lower bound on the energy, and that the step length satisfies the strong Wolfe conditions, as shown in \cite[Appendix A]{LiYu24}. Since $\L_{(\lambda^{(k)}, \rho^{(k)})}$ is a quadratic functional of the smooth energy functional $E_C(\f)$ and $E_A(\f)$  with the coefficient $\lambda^{(k)}$ and $\rho^{(k)}$, it remains a smooth energy functional. In addition, $E_C(\f) \geq 0$ and $E_A(\f) \geq 0$ are bounded below.

Furthermore, Lemma \ref{lma:lambda} shows that $\lambda^* \in [0, 1]$ and the tighter criteria for updating $\lambda^{(k)}$ in \eqref{eq:lambda up} and \eqref{eq:tight lambda} ensures $\lambda^{(k)} \in [0,1]$ during iterations. Therefore, the preconditioner $M$ in \eqref{eq:preconditioner} is a symmetric positive definite matrix. As a result, convergence is guaranteed as long as the step length satisfies the strong Wolfe conditions.
\end{proof}

Moreover, the global convergence of the augmented Lagrangian method, Algorithm \ref{alg:AugLag}, is also theoretically guaranteed, as demonstrated below:
\begin{theorem}
    Suppose the iterative sequence $\{ \bf^{(k)} \} \in \Omega$, which
    \[
    \Omega = \Big\{ \bf ~ \big|~ \bf_i \in \mathbb D ~\mbox{for}~ i = 1, \cdots, n \Big\},
    \]
    is generated by Algorithm \ref{alg:AugLag}, whose limit is $\bf^*$. Then, there exist positive numbers $a_1$, $a_2$, $s_1$ and integer $k_0$, such that for all $k \geq k_0$,
    \begin{align*}
        |\lambda^{(k+1)} - \lambda^*| &\leq a_1 \omega^{(k)} + a_2 \|\bf^{(k)} - \bf^*\|_F \\
        \big| E_A(\bf) - E_C(\bf) \big| &\leq s_1 \Big( a_1 \omega^{(k)} \mu^{(k)} + \mu^{(k)} |\lambda^{(k)} - \lambda^*| + a_2 \mu^{(k)} \|\bf^{(k)} - \bf^*\|_F \Big),
    \end{align*}
    where $\mu^{(k)} = \frac{1}{\rho^{(k)}}$, $\omega^{(k)}$ is in \eqref{eq:omega}-\eqref{eq:rho up}, and $\| \cdot \|_F$ is Frobenius norm.
    Moreover, the limit $\bf^{(k)} \rightarrow \bf^*$ is the Karush-Kuhn-Tucker point, $\lambda^{(k)} \rightarrow \lambda^*$ is the Lagrange multiplier, and the gradient of augmented Lagrangian $\nabla_{\bf} \L_A(\bf^{(k)}, \lambda^{(k)}, \rho^{(k)})$ converges to the original Lagrangian $\nabla_{\bf} \L(\bf^*, \lambda^*)$.
\end{theorem}
\begin{proof}
    The result directly follows from \cite[Theorem 4.4]{CoGT91} since the following properties hold. The function $E_C(\bf)$ and the constraint $E_A(\bf) - E_C(\bf)$ of the objective function \eqref{eq:obj_fun} are smooth for $\bf \in \Omega$, the iterative sequence $\{ \bf^{(k)} \}$ is considered to lie within a closed, bounded domain $\Omega$, as the boundary points are expressed in polar coordinates in Algorithm \ref{alg:PCG}, and the gradient of a single constraint $E_A(\bf) - E_C(\bf)$ is linearly independent at any limit point.
\end{proof}

\section{Numerical Experiments}
\label{sec:5}
This section presents the numerical results of the distortion-balancing parameterization by the augmented Lagrangian method (Algorithm \ref{alg:AugLag}). In Fig. \ref{fig:MeshModel}, we showcase the benchmark triangular mesh models and the corresponding disk-shaped distortion-balancing parameterizations. These benchmark models are sourced from reputable repositories, including the AIM@SHAPE shape repository \cite{AIM}, the Stanford 3D scanning repository \cite{Stanford}, and Sketchfab \cite{Sketchfab}. Some mesh models were modified to ensure that each triangular face contains at least one interior vertex. All experiments were conducted in MATLAB on a laptop equipped with an AMD Ryzen 9 5900HS processor and 32 GB of RAM.

To quantify local angular and area distortion, we define the relative angular and area distortions as follows:
\begin{equation}
    \mathcal{D}_\mathrm{angle}(v, [u,v,w]) = \bigg| \frac{\angle (f(u), f(v), f(w)) - \angle (u, v, w)}{ \angle (u, v, w)} \bigg|.
\label{eq:Angle_dist}
\end{equation}
and
\begin{equation}
    \mathcal{D}_\mathrm{area}(\tau) =  \bigg| \frac{|f(\tau)| / |f(\M)| - |\tau| / |\M|}{|\tau| / |\M|} \bigg|,
\label{eq:Area_dist}
\end{equation}
where $|f(\M)|$ and $|\M|$ are areas of the image and surface, respectively. For an ideal angle-preserving mapping, $\mathcal{D}_\mathrm{angle}(v, [u,v,w]) = 0$ for all $\tau = [u, v, w] \in \mathcal{F(M)}$, whereas an ideal area-preserving mapping has $\mathcal{D}_\mathrm{area}(\tau) = 0$ for all $\tau \in \mathcal{F(M)}$.

\subsection{Results of Propose Method}
Table \ref{tab:AugLM_Result} summarizes the numerical results across all benchmark triangular meshes, including computational time (in seconds), conformal energy $E_C$ in \eqref{eq:Ec}, energy different $|E_A - E_C|$ in \eqref{eq:Ea}, approximated Lagrange multiplier $\lambda$, the number of folding triangles, and the mean and standard deviation of angle and area distortion $\mathcal{D}_\mathrm{angle}$ and $\mathcal{D}_\mathrm{area}$ in \eqref{eq:Angle_dist} and \eqref{eq:Area_dist}.

The energy difference between conformal and authalic energies indicates a well-balanced trade-off in global angular and area distortion. This is visualized in Fig. \ref{fig:Energy}, which illustrates the progression of conformal and authalic energies during the iterations. The balance in energy is also reflected in the local angular and area distortions, as evidenced by the similarities in the mean and standard deviation of $\mathcal{D}_\mathrm{angle}$ and $\mathcal{D}_\mathrm{area}$ with their distributions shown in Fig. \ref{fig:Distortion}. The approximated Lagrange multiplier $\lambda$ suggests that different meshes require different value of $\lambda$. Additionally, our proposed method demonstrates efficiency, achieving computational times of under one minute for all mesh models. Most importantly, the absence of folding triangles in all benchmark meshes confirms the bijectivity of the parameterizations produced by our method.

\begin{figure}[]
\centering
\resizebox{\textwidth}{!}{
\begin{tabular}{cccccccc}
\cline{1-2}\cline{4-5}\cline{7-8}
\multicolumn{2}{c}{Chinese Lion} && \multicolumn{2}{c}{Femur} && \multicolumn{2}{c}{Stanford Bunny} \\
$\F(\M) = 34,421$ & $\V(\M) = 17,334$ && $\F(\M) = 43,301$ & $\V(\M) = 21,699$ && $\F(\M) = 62,946$ & $\V(\M) = 31,593$ \\
\includegraphics[height=3cm]{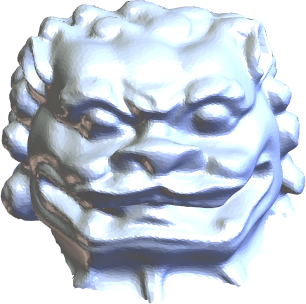} &
\includegraphics[height=3cm]{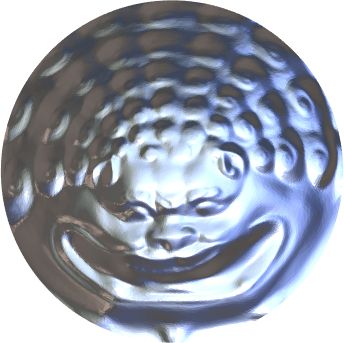} &&
\includegraphics[height=3cm]{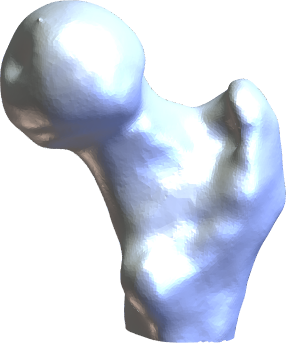} &
\includegraphics[height=3cm]{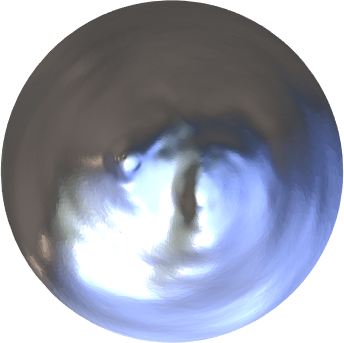} &&
\includegraphics[height=3cm]{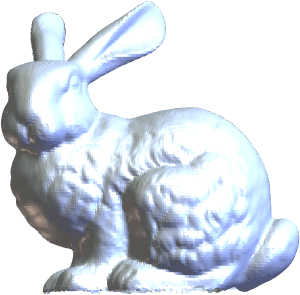} & 
\includegraphics[height=3cm]{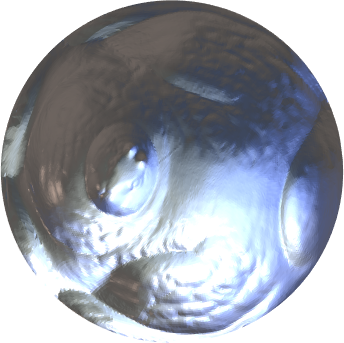} 
\\ \Cline{1.2pt}{1-2}\Cline{1.2pt}{4-5}\Cline{1.2pt}{7-8}
\multicolumn{2}{c}{Max Planck} && \multicolumn{2}{c}{Human Brain} && \multicolumn{2}{c}{Left Hand} \\
$\F(\M) = 82,977$ & $\V(\M) = 41,588$ && $\F(\M) = 96,811$ & $\V(\M) = 48,463$ && $\F(\M) = 105,780$ & $\V(\M) = 53,011$ \\
\includegraphics[height=3cm]{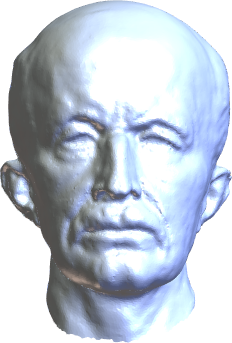} &
\includegraphics[height=3cm]{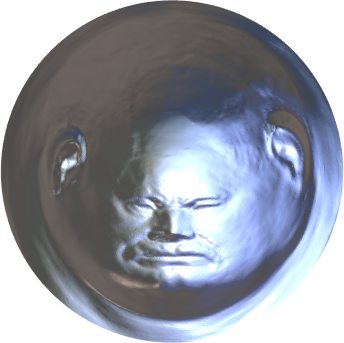} &&
\includegraphics[height=3cm]{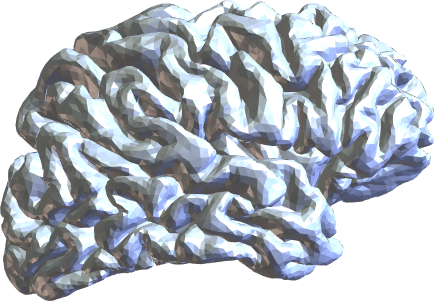} &
\includegraphics[height=3cm]{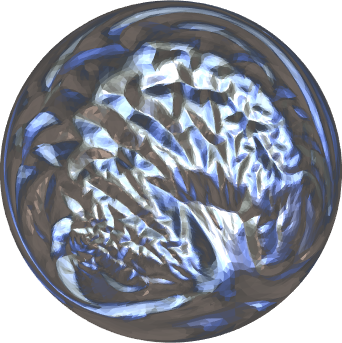} &&
\includegraphics[height=3cm]{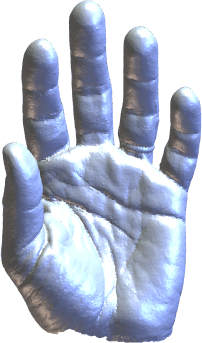} &
\includegraphics[height=3cm]{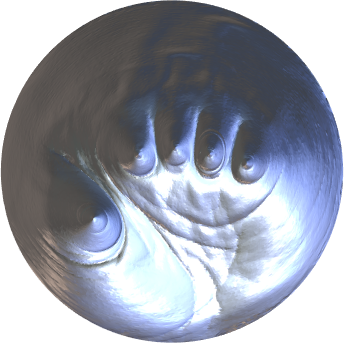} 
\\ \Cline{1.2pt}{1-2}\Cline{1.2pt}{4-5}\Cline{1.2pt}{7-8}
\multicolumn{2}{c}{Knit Cap Man} && \multicolumn{2}{c}{Bimba Statue} && \multicolumn{2}{c}{Buddha} \\
$\F(\M) = 118,849$ & $\V(\M) = 59,561$ && $\F(\M) = 433,040$ & $\V(\M) = 216,873$ && $\F(\M) = 945,722$ & $\V(\M) = 473,362$ \\
\includegraphics[height=3cm]{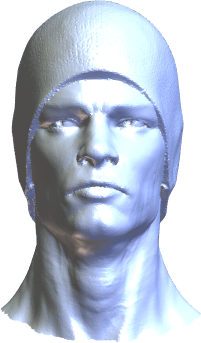} & 
\includegraphics[height=3cm]{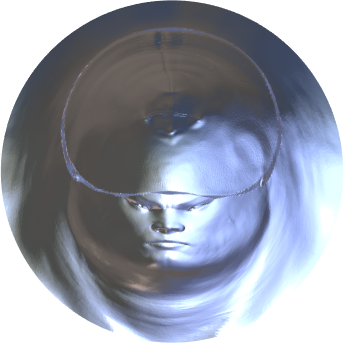} &&
\includegraphics[height=3cm]{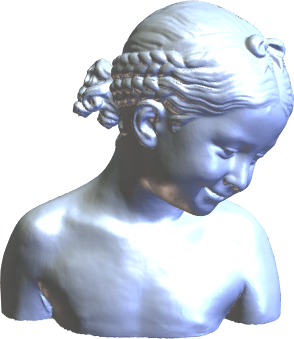} &
\includegraphics[height=3cm]{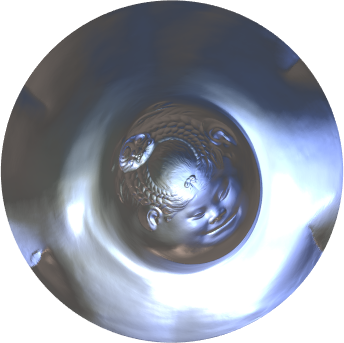} &&
\includegraphics[height=3cm]{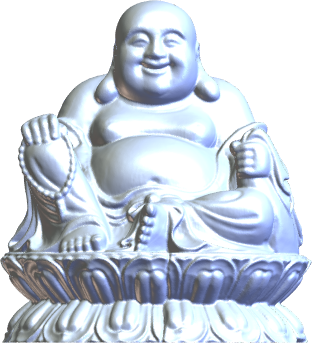} & 
\includegraphics[height=3cm]{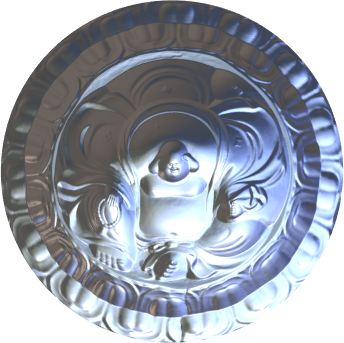} 
\\ \cline{1-2}\cline{4-5}\cline{7-8}
\end{tabular}
}
\caption{The benchmark triangular mesh models and associated disk-shaped distortion-balancing parameterization produced by Algorithm \ref{alg:AugLag}.}
\label{fig:MeshModel}
\end{figure}

\begin{table}[]
\centering
\caption{The numerical results of Algorithm \ref{alg:AugLag}. Time: computational time costs in seconds; $E_C$ : conformal energy \eqref{eq:Ec}; $E_A$: authalic energy \eqref{eq:Ea}; $\lambda$: approximated Lagrange multipler; Iter. no.: number of iterations; $\#$Foldings: number of folding triangles; $\mathcal{D}_\mathrm{angle}$: angular distortion \eqref{eq:Angle_dist}; $\mathcal{D}_\mathrm{area}$: area distortion \eqref{eq:Area_dist}.
}
\label{tab:AugLM_Result}
\resizebox{\textwidth}{!}{
\begin{tabular}{lrcccrccccc}
\hline 
\multirow{3}{*}{Model name} & \multicolumn{10}{c}{Augmented Lagrangian method, Algorithm \ref{alg:AugLag}.} \\
\cline{2-11}
& Time      & \multirow{2}{*}{$E_C$} & Energy diff. & \multirow{2}{*}{$\lambda$} & Iter. & $\#$Fold- & \multicolumn{2}{c}{$\mathcal{D}_\mathrm{angle}$} & \multicolumn{2}{c}{$\mathcal{D}_\mathrm{area}$} \\ 
& {(secs.)} &                        & $|E_A - E_C|$&                            & no.   &  ings      & Mean           & SD             & Mean        & SD \\
\hline 
Chinese Lion  & 1.34 & $4.75\times 10^{-1}$ & $7.34\times 10^{-6}$ & 0.38 & 9  & 0 & 0.20 & 0.17 & 0.33 & 0.20 \\ 
Femur         & 2.62 & $2.87\times 10^{~0~}$& $8.34\times 10^{-7}$ & 0.39 & 9  & 0 & 0.27 & 0.32 & 0.81 & 0.51 \\ 
Stanford Bunny& 4.99 & $1.10\times 10^{~0~}$& $3.33\times 10^{-6}$ & 0.47 & 9  & 0 & 0.26 & 0.27 & 0.48 & 0.33 \\
Max Planck    & 7.30 & $1.65\times 10^{~0~}$& $6.76\times 10^{-7}$ & 0.31 & 8  & 0 & 0.25 & 0.27 & 0.56 & 0.43 \\
Human Brain   & 10.10& $1.92\times 10^{~0~}$& $1.82\times 10^{-6}$ & 0.26 & 8  & 0 & 0.25 & 0.25 & 0.58 & 0.54 \\
Left Hand     & 14.75& $2.23\times 10^{~0~}$& $6.64\times 10^{-6}$ & 0.56 & 10 & 0 & 0.33 & 0.33 & 0.78 & 0.32 \\
Knit Cap Man  & 8.87 & $2.12\times 10^{~0~}$& $7.83\times 10^{-6}$ & 0.37 & 8  & 0 & 0.28 & 0.30 & 0.69 & 0.44 \\
Bimba Statue  & 11.42& $1.69\times 10^{~0~}$& $5.47\times 10^{-7}$ & 0.49 & 11 & 0 & 0.30 & 0.30 & 0.70 & 0.22 \\
Buddha        & 36.43& $7.62\times 10^{-1}$ & $7.11\times 10^{-7}$ & 0.30 & 12 & 0 & 0.25 & 0.21 & 0.41 & 0.27 \\
\hline 
\end{tabular}
}
\end{table}

\begin{figure}[h]
\centering
\resizebox{\textwidth}{!}{
\begin{tabular}{ccc}
Chinese Lion & Femur & Stanford Bunny \\
\includegraphics[width=4.5cm]{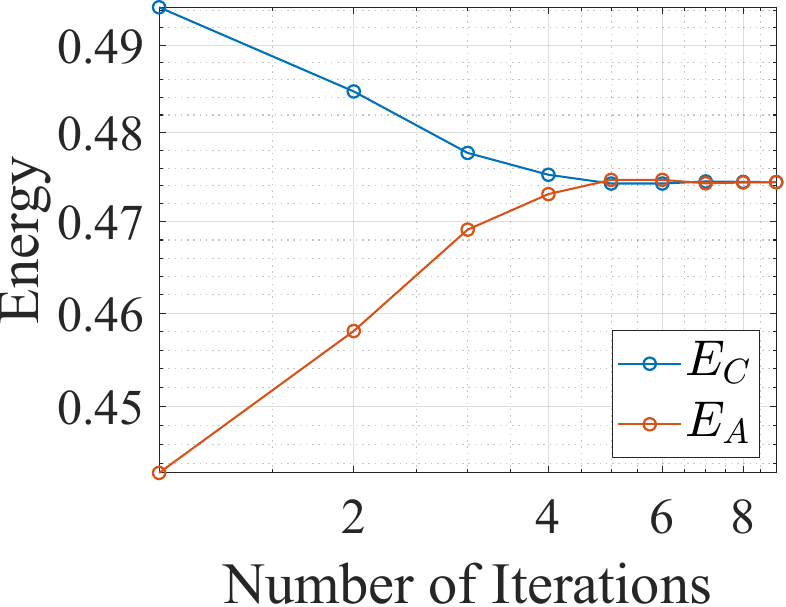} &
\includegraphics[width=4.5cm]{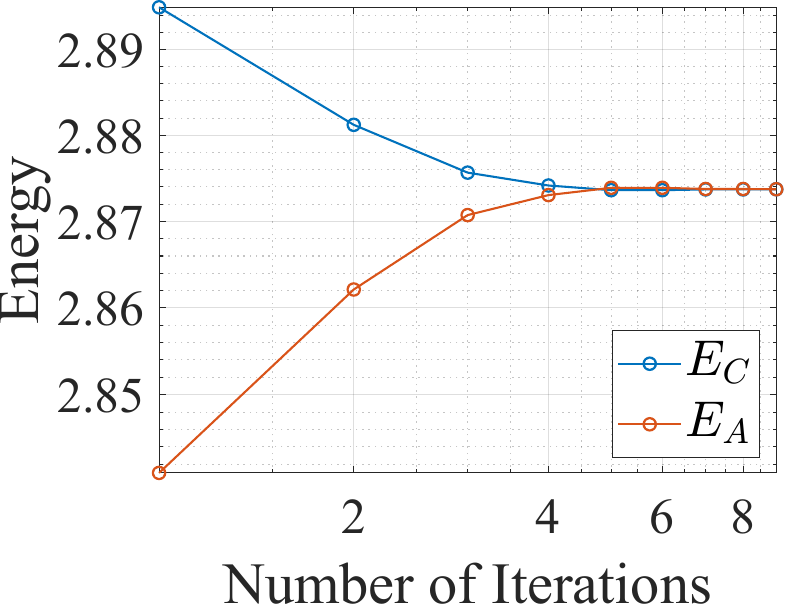} &
\includegraphics[width=4.5cm]{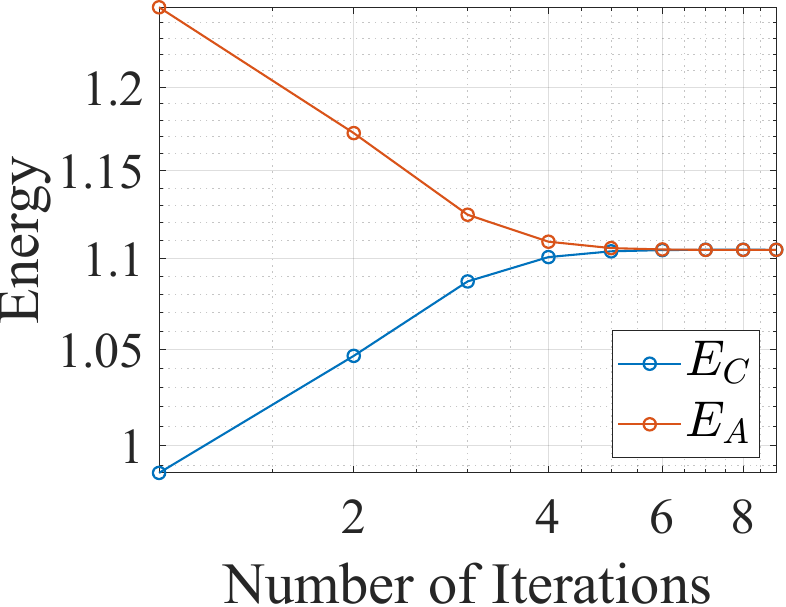} \\ [0.5cm]
Max Planck & Human Brain & Left Hand \\
\includegraphics[width=4.5cm]{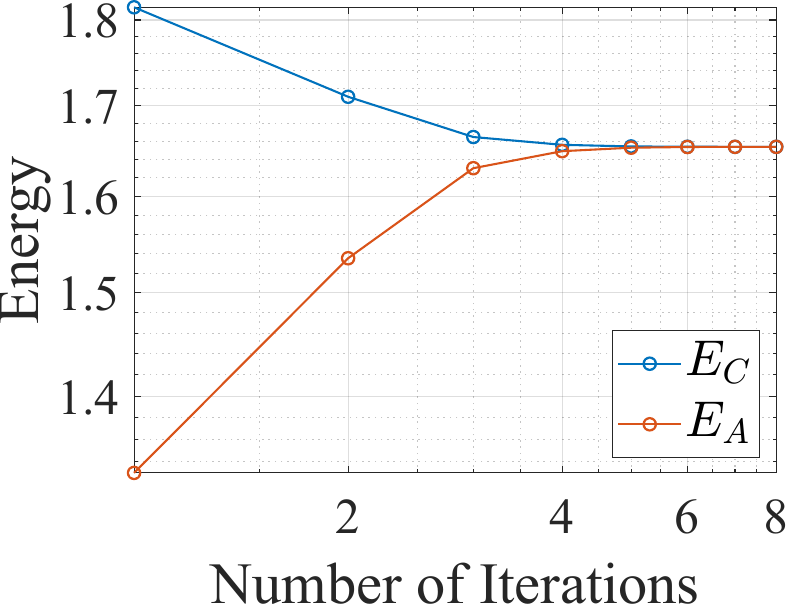} &
\includegraphics[width=4.5cm]{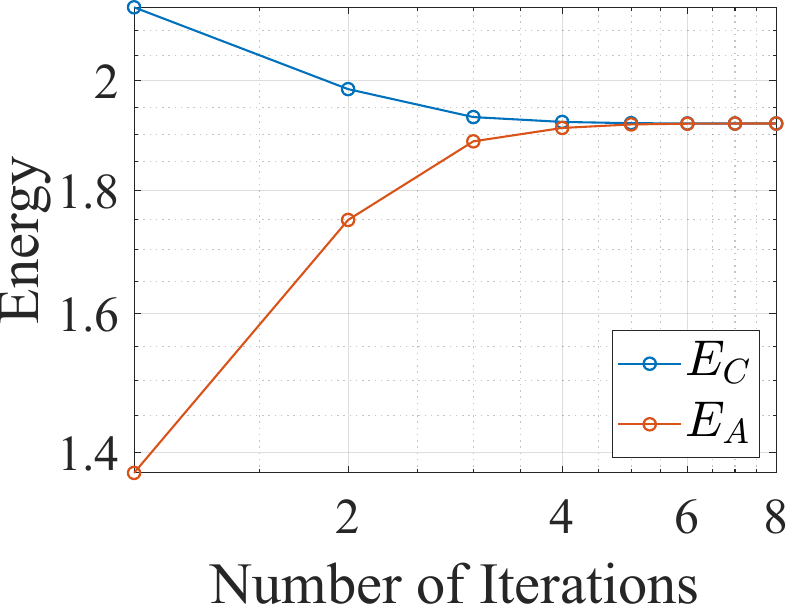} &
\includegraphics[width=4.5cm]{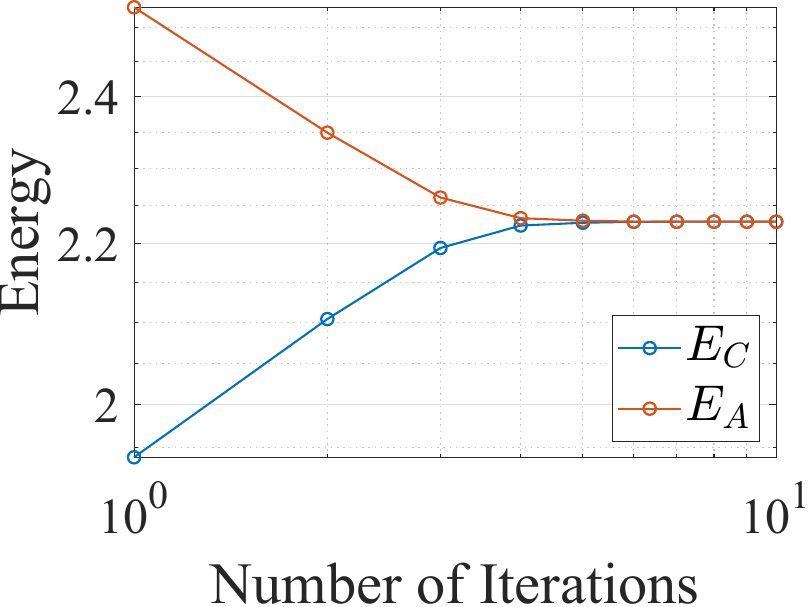} \\ [0.5cm]
Knit Cap Man & Bimba Statue & Buddha \\
\includegraphics[width=4.5cm]{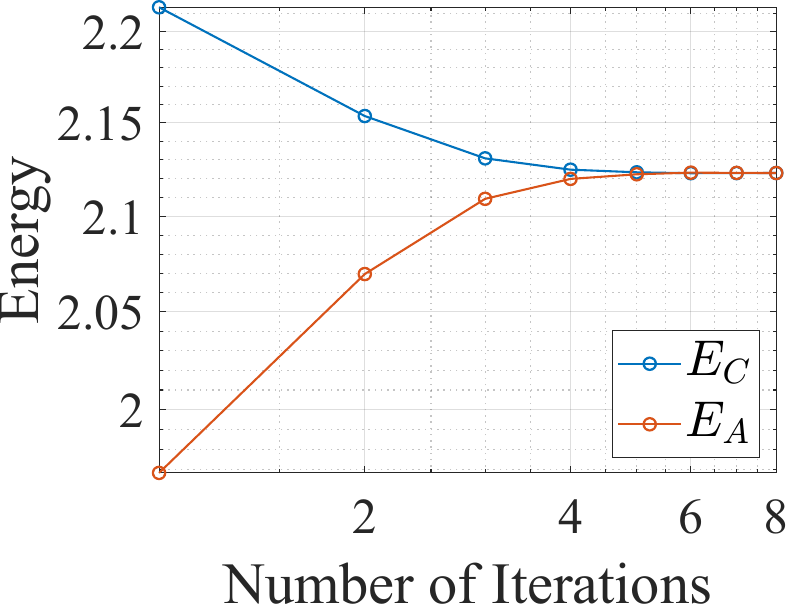} &
\includegraphics[width=4.5cm]{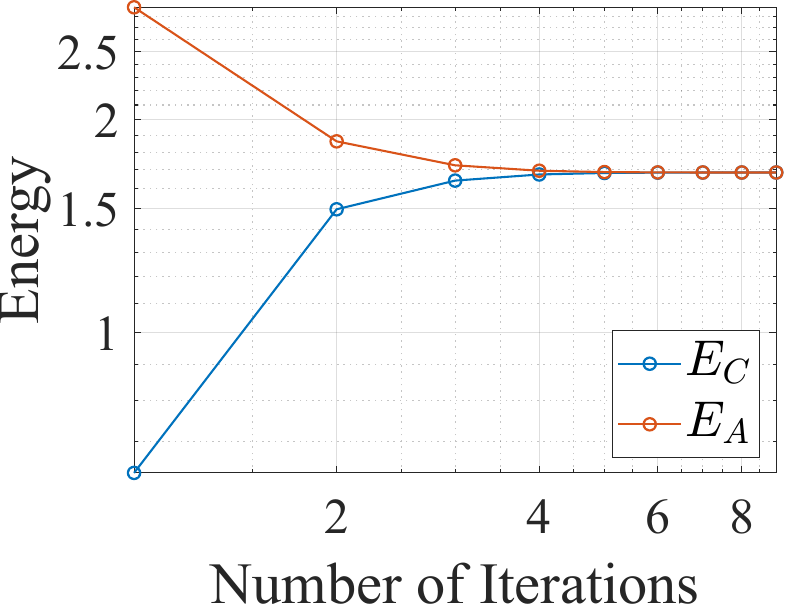} &
\includegraphics[width=4.5cm]{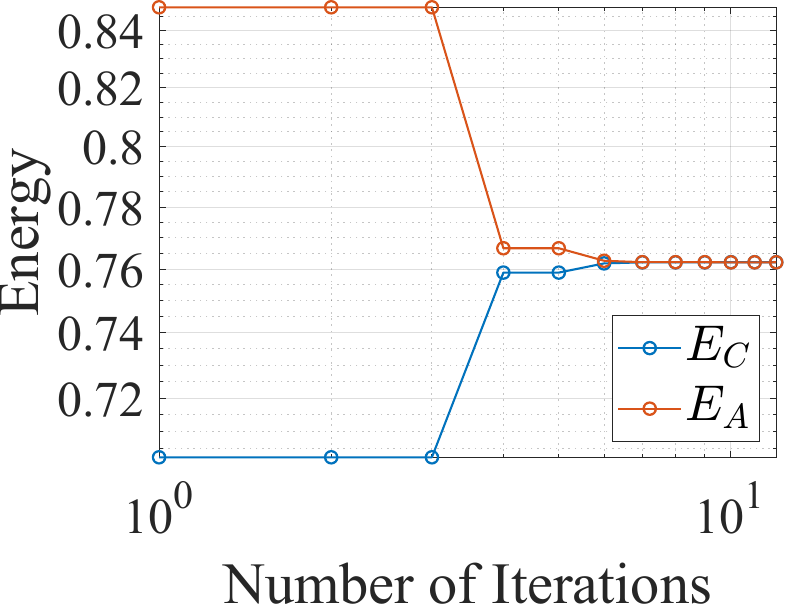} \\
\end{tabular}
}
\caption{The conformal \eqref{eq:Ec} and authalic \eqref{eq:Ea} energy during iterations by Algorithm \ref{alg:AugLag} among all benchmark triangular meshes.}
\label{fig:Energy}
\end{figure}

\begin{figure}[]
\centering
\resizebox{\textwidth}{!}{
\begin{tabular}{ccc}
Chinese Lion & Femur & Stanford Bunny \\
\begin{overpic}[width=4.5cm]{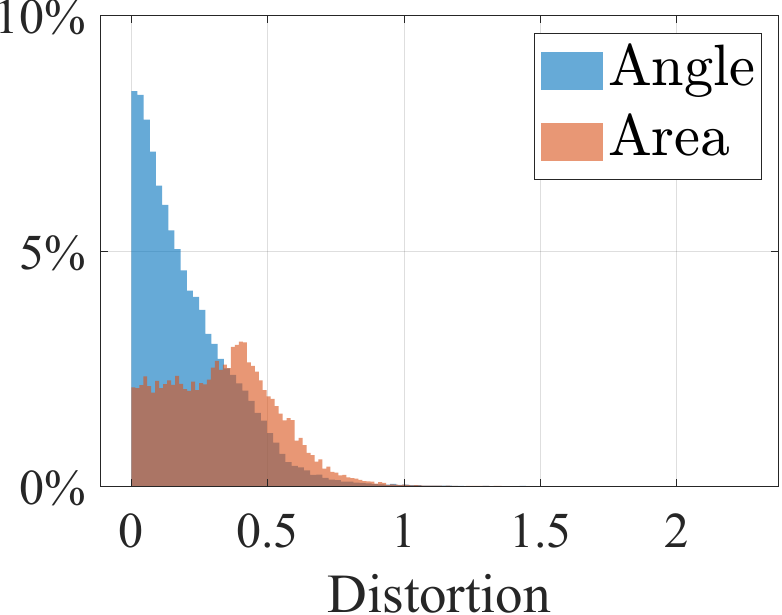}
    \put(55,20){\includegraphics[width=1.2cm]{ChineseLion_Mesh.png}}
\end{overpic} &
\begin{overpic}[width=4.5cm]{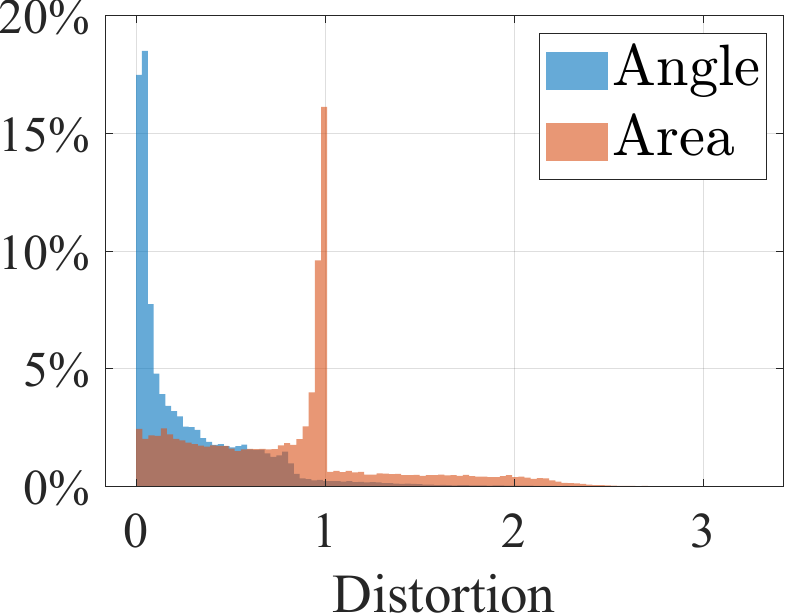}
    \put(60,20){\includegraphics[width=1.0cm]{Femur_Mesh.png}}
\end{overpic} &
\begin{overpic}[width=4.5cm]{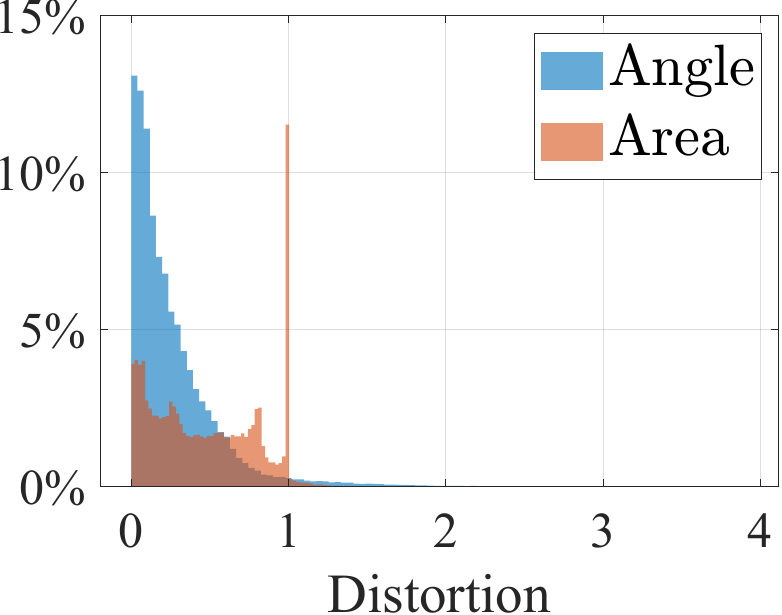}
    \put(55,20){\includegraphics[width=1.25cm]{StanfordBunny_Mesh.png}}
\end{overpic} \\[0.5cm]
Max Planck & Human Brain & Left Hand \\
\begin{overpic}[width=4.5cm]{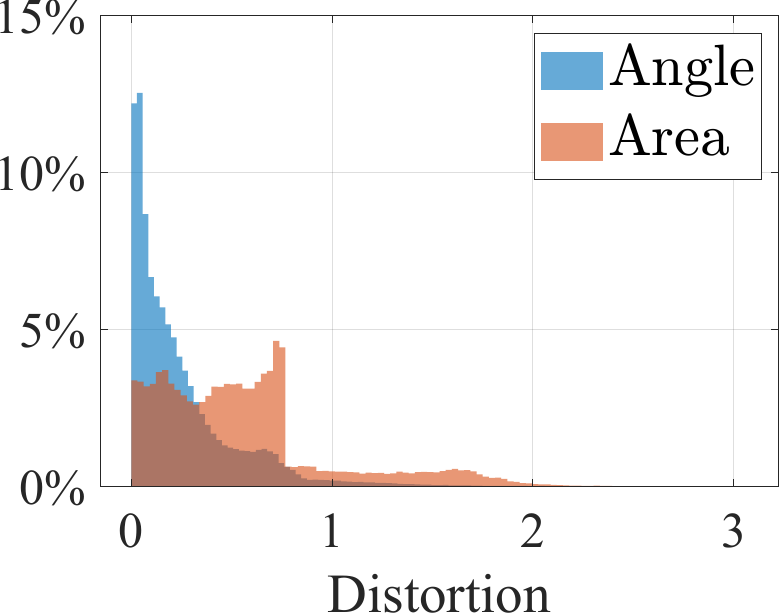}
    \put(58,18){\includegraphics[width=0.95cm]{MaxPlanck_Mesh.png}}
\end{overpic} &
\begin{overpic}[width=4.5cm]{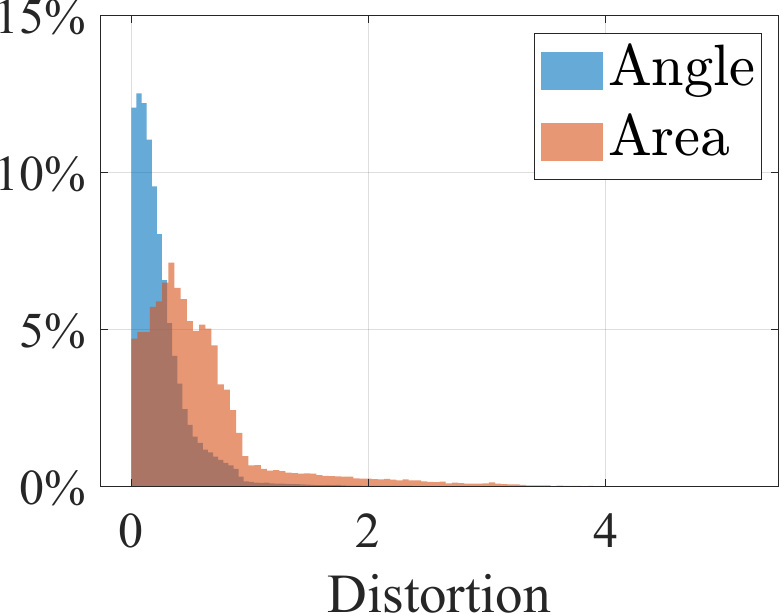}
    \put(55,20){\includegraphics[width=1.45cm]{HumanBrain_Mesh.png}}
\end{overpic} &
\begin{overpic}[width=4.5cm]{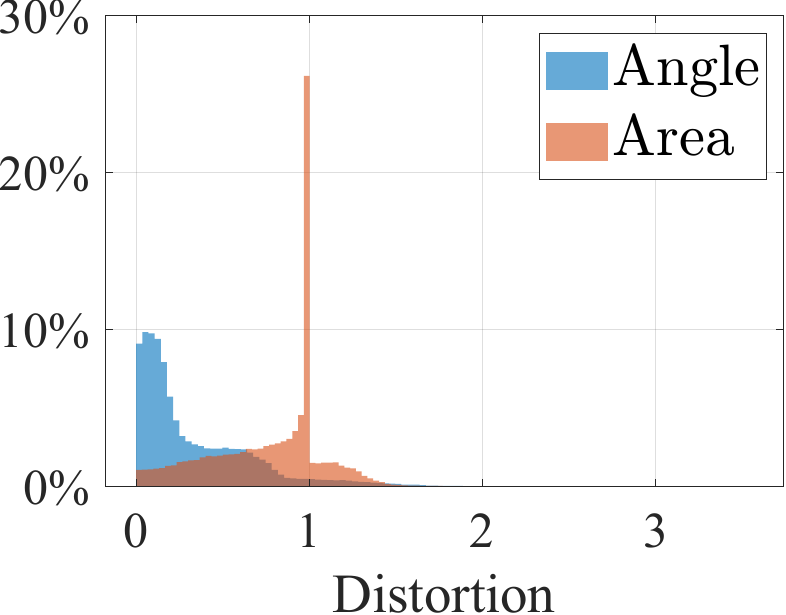}
    \put(60,20){\includegraphics[width=0.8cm]{LeftHand_Mesh.png}}
\end{overpic} \\[0.5cm]
Knit Cap Man & Bimba Statue & Buddha \\
\begin{overpic}[width=4.5cm]{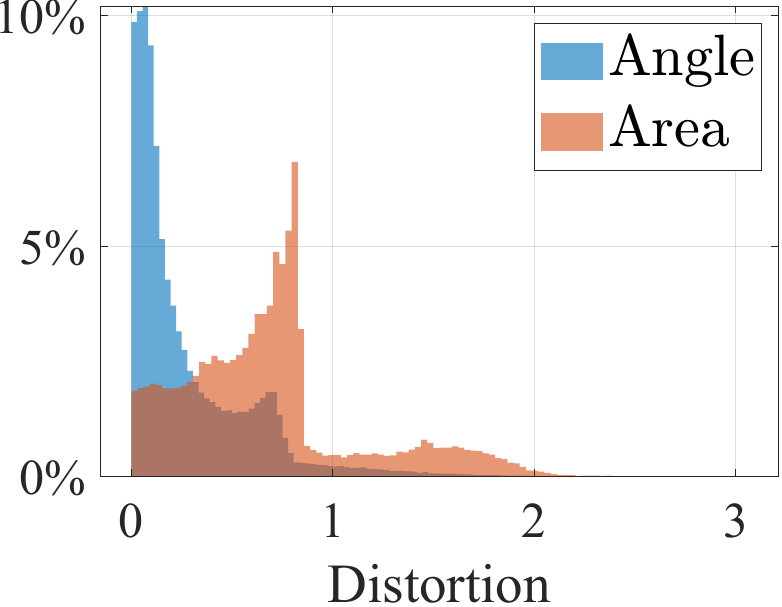}
    \put(60,18){\includegraphics[width=0.82cm]{KnitCapMan_Mesh.png}}
\end{overpic} &
\begin{overpic}[width=4.5cm]{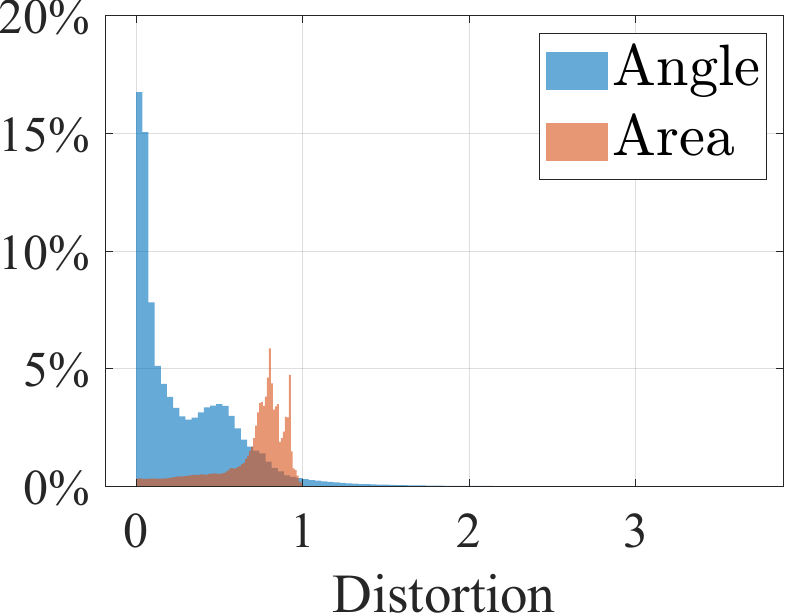}
    \put(58,18){\includegraphics[width=1.2cm]{BimbaStatue_Mesh.png}}
\end{overpic} &
\begin{overpic}[width=4.5cm]{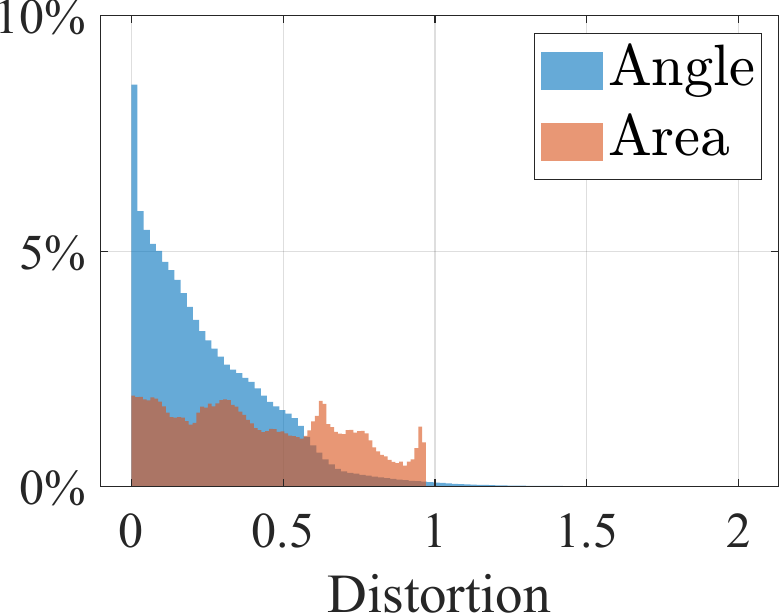}
    \put(55,18){\includegraphics[width=1.25cm]{Buddha_Mesh.png}}
\end{overpic} \\
\end{tabular}
}
\caption{The histogram of angular \eqref{eq:Angle_dist} and area \eqref{eq:Area_dist} distortion of the parameterization produced by Algorithm \ref{alg:AugLag} among all benchmark triangular meshes.}
\label{fig:Distortion}
\end{figure}

\subsection{Comparison to the Previous Method}
\label{sec:5.2}
In the previous work \cite{Yueh23}, to obtain the distortion-balancing mapping, it minimizes the balance energy  
\begin{equation} \label{eq:Eb}
E_B(\bf) = \frac{1}{2} (E_C(\bf) + E_S(\bf))
\end{equation}
using a fixed-point method, and the boundary points $\bf_\B$ are selected via arc length parameterization of $\partial \M$ and kept fixed. We can regard \eqref{eq:Eb} as the Lagrange function in \eqref{eq:Lag} as $\lambda = 0.5$. The numerical results with tolerance set by energy deficit less than $10^{-6}$ are presented in Table \ref{tab:BEM_Result}. The zeros folding triangles show that all parameterizations provided are bijective. The difference between the conformal and authalic energy suggests $\lambda \neq 0.5$ in most cases. This discrepancy is also reflected in the angular and area distortion, as defined by \eqref{eq:Angle_dist} and \eqref{eq:Area_dist}.

To directly compare the fixed-point method \cite{Yueh23} with our proposed method, we use the accuracy metric of the mapping produced by Algorithm \ref{alg:AugLag} divided by the accuracy metric of the mapping provided by the fixed-point method, including the conformal and authalic energy, and the mean angular and area distortion (see Fig. \ref{fig:BEM_comparison}). The results indicate that the fixed-point method tends to favor either angle-preserving or area-preserving properties and the degree of bias varies across different triangular meshes.

\begin{table}[]
\centering
\caption{The numerical results of the fixed-point method \cite{Yueh23} with the tolerance set by energy deficit less than $10^{-6}$ and fix the boundary point by the arc length parameterization of $\partial \M$. Time: computational time costs in seconds; $E_C$ : conformal energy \eqref{eq:Ec}; $E_A$: authalic energy \eqref{eq:Ea}; Iter. no.: number of iterations; $\#$Foldings: number of folding triangles; $\mathcal{D}_\mathrm{angle}$: angular distortion \eqref{eq:Angle_dist}; $\mathcal{D}_\mathrm{area}$: area distortion \eqref{eq:Area_dist}.
}
\label{tab:BEM_Result}
\resizebox{\textwidth}{!}{
\begin{tabular}{lrccrccccc}
\hline 
\multirow{3}{*}{Model name} & \multicolumn{9}{c}{the fixed-point method \cite{Yueh23}.} \\
\cline{2-10}
& Time      & \multirow{2}{*}{$E_C$} & \multirow{2}{*}{$E_A$} & Iter. & $\#$Fold- & \multicolumn{2}{c}{$\mathcal{D}_\mathrm{angle}$} & \multicolumn{2}{c}{$\mathcal{D}_\mathrm{area}$} \\ 
& {(secs.)} &                        &                        & no.   &  ings      & Mean           & SD             & Mean        & SD \\
\hline 
Chinese Lion  & 0.66 & $4.26\times 10^{-1}$ & $6.12\times 10^{-1}$ &13 & 0 & 0.18 & 0.16 & 0.38 & 0.22 \\ 
Femur         & 1.84 & $2.47\times 10^{~0~}$& $3.62\times 10^{~0~}$&28 & 0 & 0.24 & 0.29 & 0.90 & 0.58 \\ 
Stanford Bunny& 2.35 & $8.31\times 10^{-1}$ & $1.65\times 10^{~0~}$&22 & 0 & 0.20 & 0.22 & 0.58 & 0.41 \\
Max Planck    & 4.59 & $1.79\times 10^{~0~}$& $1.58\times 10^{~0~}$&33 & 0 & 0.26 & 0.28 & 0.54 & 0.42 \\
Human Brain   & 8.49 & $2.23\times 10^{~0~}$& $1.30\times 10^{~0~}$&57 & 0 & 0.28 & 0.28 & 0.48 & 0.43 \\
Left Hand     & 5.28 & $1.43\times 10^{~0~}$& $3.42\times 10^{~0~}$&29 & 0 & 0.24 & 0.25 & 0.92 & 0.50 \\
Knit Cap Man  & 5.35 & $1.94\times 10^{~0~}$& $2.57\times 10^{~0~}$&26 & 0 & 0.26 & 0.29 & 0.75 & 0.49 \\
Bimba Statue  & 20.13& $4.72\times 10^{-1}$ & $3.18\times 10^{~0~}$&20 & 0 & 0.16 & 0.17 & 0.96 & 0.31 \\
Buddha        & 36.11& $5.51\times 10^{-1}$ & $1.11\times 10^{~0~}$&13 & 0 & 0.19 & 0.18 & 0.51 & 0.31 \\
\hline 
\end{tabular}
}
\end{table}

\begin{figure}[]
\centering
\resizebox{\textwidth}{!}{
\begin{tabular}{cc}
\includegraphics[height=4.5cm]{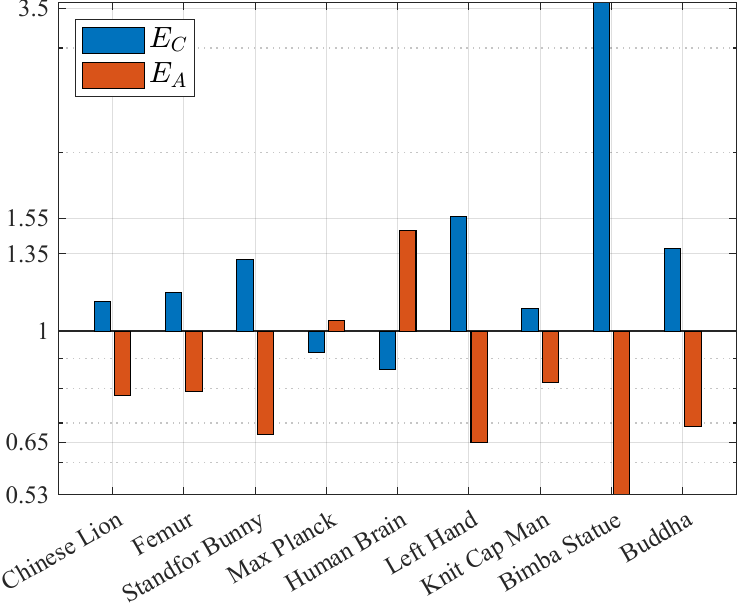} &
\includegraphics[height=4.5cm]{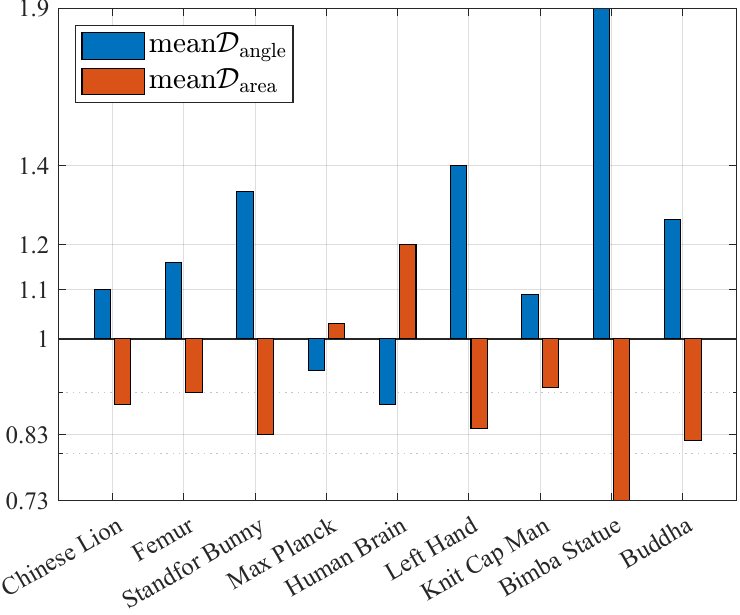} 
\end{tabular}
}
\caption{The ratio of accuracy metric of mapping produced by Algorithm \ref{alg:AugLag} divided by the accuracy metric of mapping provided by the fixed-point method \cite{Yueh23}, among all benchmark triangular meshes. $E_C$ : conformal energy \eqref{eq:Ec}; $E_A$: authalic energy \eqref{eq:Ea}; mean$\mathcal{D}_\mathrm{angle}$: mean of angular distortion \eqref{eq:Angle_dist}; mean$\mathcal{D}_\mathrm{area}$: mean of area distortion \eqref{eq:Area_dist}.}
\label{fig:BEM_comparison}
\end{figure}

\section{Application in Geometry Image}
\label{sec:6}
A geometry image \cite{GuGH02} is an RGB image that encodes the geometry of a surface by storing information, such as vertex positions and normal vectors, in each pixel, enabling surface reconstruction. This is achieved by parametrizing the surface with the planer rectangular domain. 

We can construct the geometry image using the square-shaped parameterization described in Sect. \ref{sec:3.4}. After parameterization, the geometry image stores the vertex information on the 3D face in each corresponding pixel of the 2D image, with the $x$, $y$, and $z$-coordinate of the vertices represented in the red, green, and blue channels of the image (see Fig. \ref{fig:GeoImg}). To reconstruct the triangular mesh of the 3D surface from the geometry image, a structured triangular mesh is generated on the 2D image by inserting a center point into each quad element. Each vertex from the 2D image is then mapped to its corresponding vertex on the 3D surface.

The mapping used to construct a geometry image affects the quality of the reconstructed surface. Conformal mapping preserves the shape of triangles but may introduce defects in the overall reconstruction. In contrast, authalic mapping provides a more uniform vertex distribution, which benefits the reconstruction but compromises the quality of triangles. We hypothesize that distortion-balancing mapping can offer a compromise, balancing the strengths of both conformal and authalic mapping. This is examined by numerical experiments discussed in the next section.

\subsection{Numerical Results of Reconstruction}
To evaluate the efficacy of surface reconstruction by different parameterizations in the same geometry image, we use the following metrics,
\begin{align}
    d_{\textrm{angle}} &= \min \Big( |\angle [v_i, v_j, v_k] - 45|, ~~ |\angle [v_i, v_j, v_k] -90| \Big), \label{eq:angle_d}\\
    d_{\textrm{area}} &= \Big( \big| |\tau| - \mean_{\tau \in \mathcal{F(M)}} ( |\tau|) \big| \Big) / \mean_{\tau \in \mathcal{F(M)}} ( |\tau|) \label{eq:area_d}
\end{align}
because the triangular mesh generated in the image is by inserting a center point into each quad element. If $f$ is conformal, $d_{\textrm{angle}} = 0$, by contrast, if $f$ is area-preserving, $d_{\textrm{area}} = 0$. 

Fig. \ref{fig:Brain Reconstruct} shows the surface reconstruction of the Human Brain triangular mesh using a $100 \times 100$ geometry image. The conformal and authalic mapping is by the fixed-point method \cite{Yueh23} and the distortion-balancing is by Algorithm \ref{alg:AugLag} with modification in Section \ref{sec:3.4}. We also provide a histogram illustrating the distribution of angles in degree and the areas normalized by the mean area of the reconstructed surface. Additionally, the mean and standard deviation for the metrics $d_{\textrm{angle}}$ and $d_{\textrm{area}}$ are also reported.

As anticipated, conformal mapping results in a wide variation in areas of different triangles, where most triangles have small areas, but a few outliers are 60 times larger than the mean, which causes the reconstructed surface to lose some detail. In contrast, authalic mapping achieves a uniform distribution of triangle areas, preserving surface detail but producing many visibly thin triangles. The distortion-balancing mapping, however, appears to reach balanced angle and area distributions,  maintaining detail and geometric integrity.

\begin{figure}[]
\center
\resizebox{\textwidth}{!}{
\begin{tabular}{ccccccc}
\includegraphics[height=3cm]{ChineseLion_Mesh.png} &
\raisebox{1.2cm}{$\longrightarrow$} & 
\includegraphics[height=3cm]{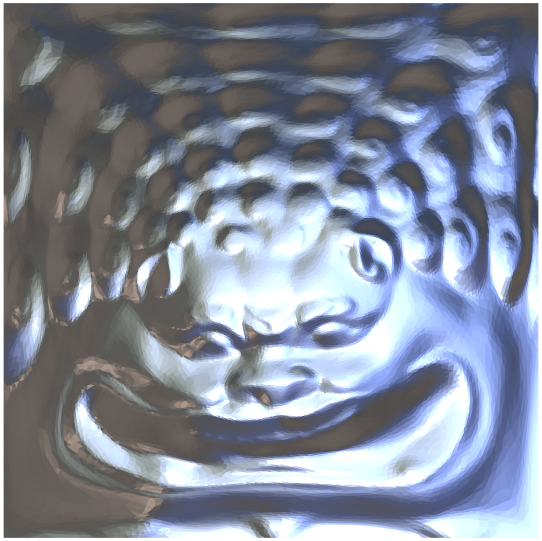} & \raisebox{1.2cm}{$\longrightarrow$} &
\begin{overpic}[width=3cm]{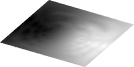}
    \put(0,20){\includegraphics[width=3cm]{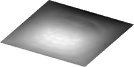}}
    \put(0,40){\includegraphics[width=3cm]{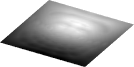}}
\end{overpic} & 
\raisebox{1.2cm}{$\longrightarrow$} & 
\includegraphics[height=3cm]{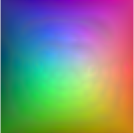}\\
Triangular Mesh & & Parameterization & & Grayscale images in each axis & & Combine to an RGB image
\end{tabular}
}
\caption{The process of construction in geometry image.
We first parameterized a simplicial surface onto the unit square and then created an image such that each pixel stored $x$, $y$, and $z$ coordinates of the surface by red, green, and blue channels,
}
\label{fig:GeoImg}
\end{figure}

\begin{figure}[]
\centering
\resizebox{\textwidth}{!}{
\begin{tabular}{cccccc}
\multicolumn{2}{c}{Balanced} & \multicolumn{2}{c}{Conformal} & \multicolumn{2}{c}{Authalic} \\
\multicolumn{2}{c}{\includegraphics[height=4.0cm]{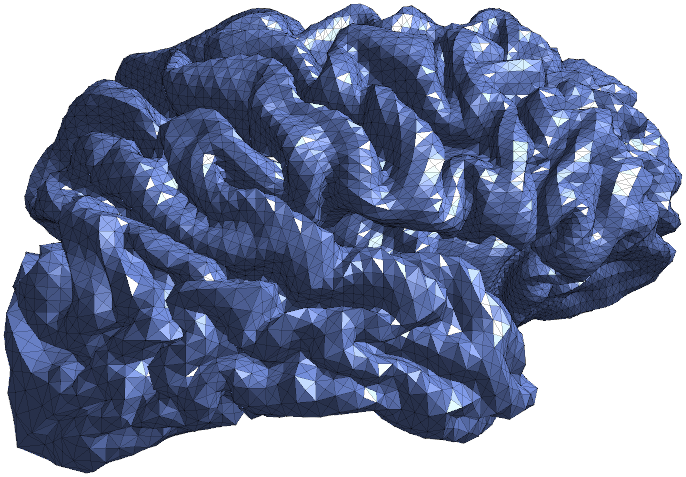}} & 
\multicolumn{2}{c}{\includegraphics[height=4.0cm]{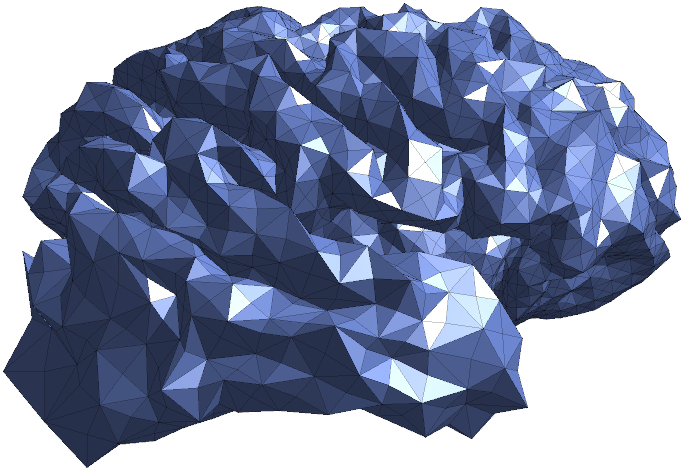}} & 
\multicolumn{2}{c}{\includegraphics[height=4.0cm]{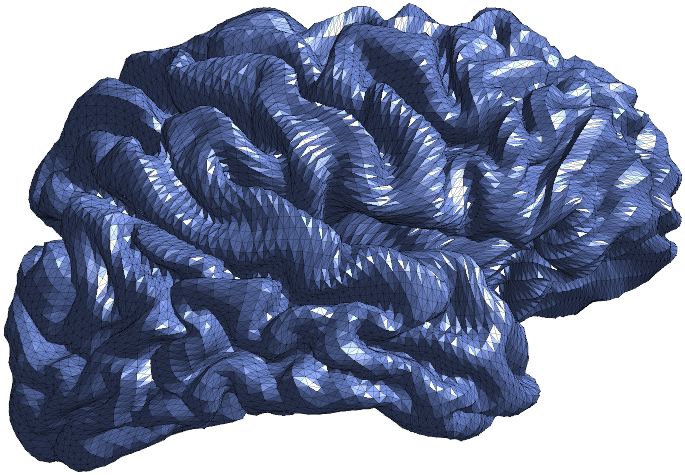}} \\
\includegraphics[height=2.5cm]{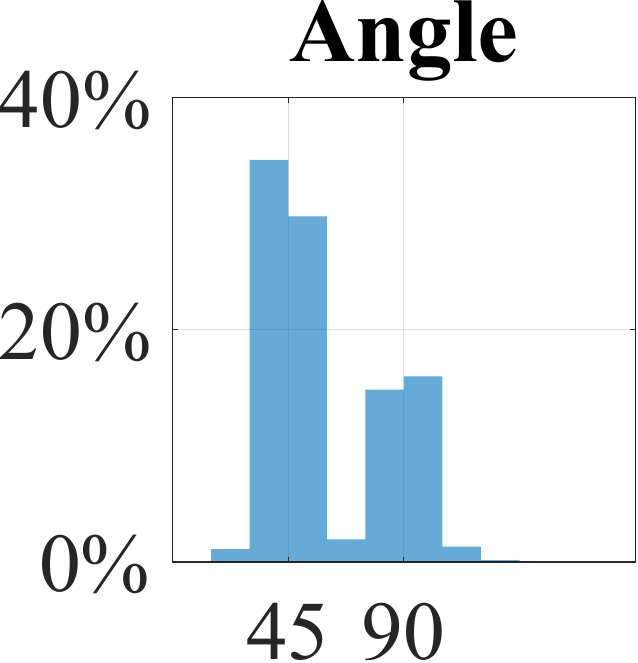} &
\includegraphics[height=2.5cm]{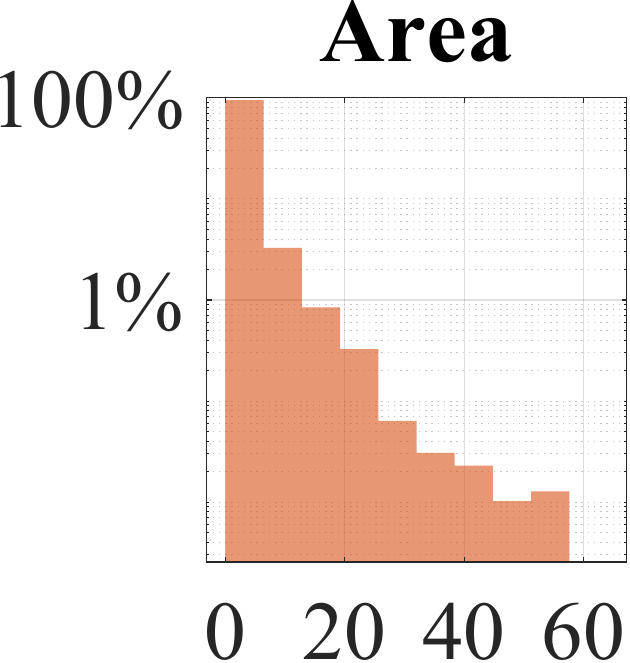} &
\includegraphics[height=2.5cm]{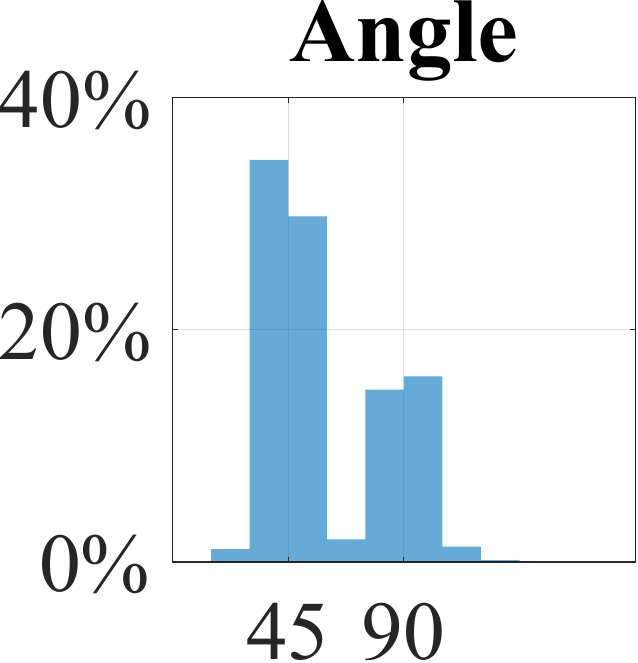} &
\includegraphics[height=2.5cm]{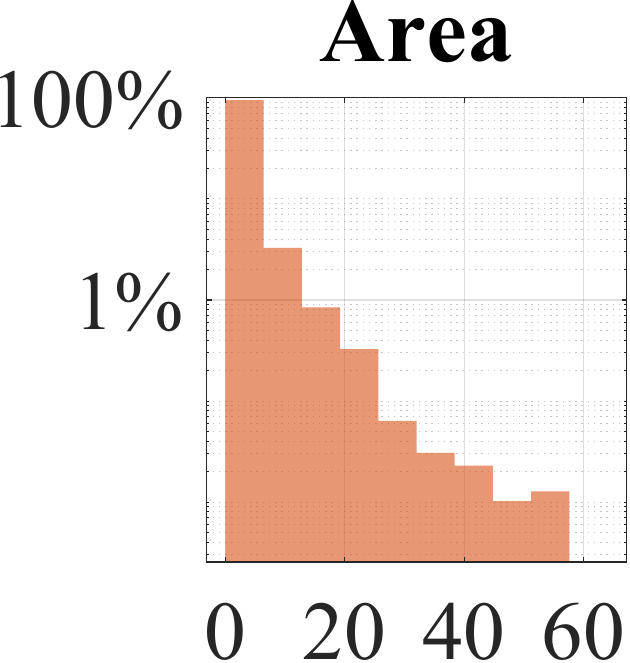} &
\includegraphics[height=2.5cm]{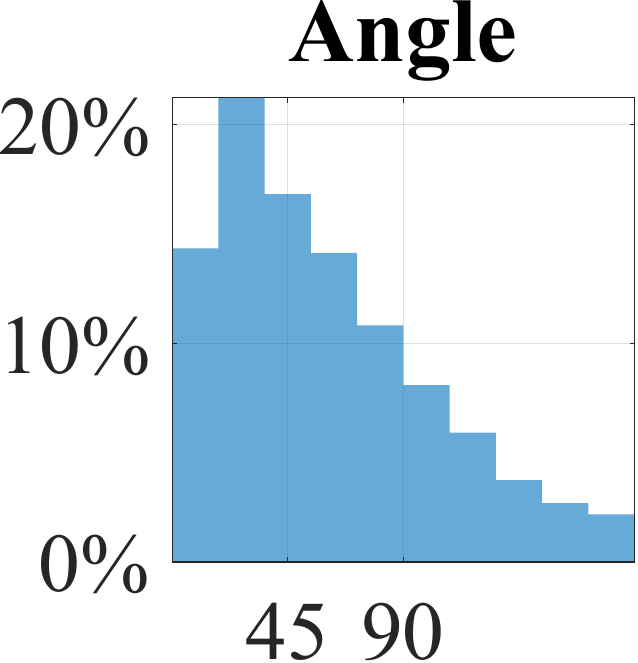} & 
\includegraphics[height=2.5cm]{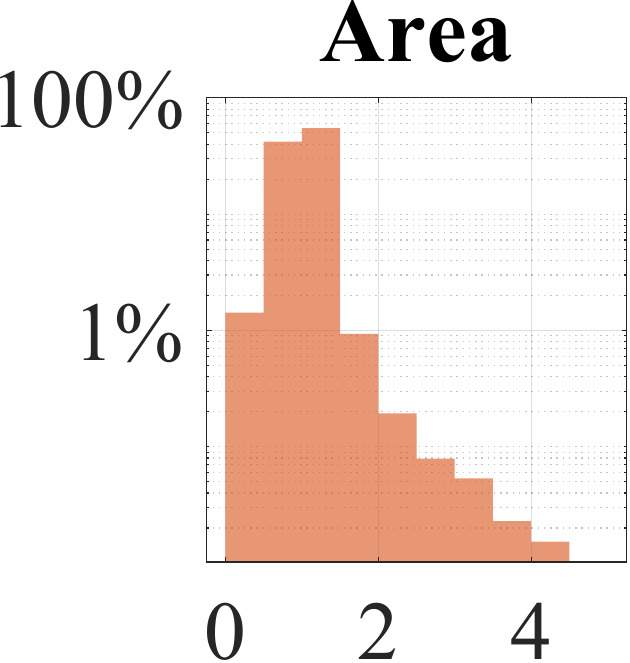} \\
$\mean d_{\textrm{angle}} = 16.3$ & $\mean d_{\textrm{area}} = 0.6$ &
$\mean d_{\textrm{angle}} = 3.6$ & $\mean d_{\textrm{area}} = 1.4$ &
$\mean d_{\textrm{angle}} = 20.4$ & $\mean d_{\textrm{area}} = 0.1$ \\
$\SD d_{\textrm{angle}} = 14.8$ & $\SD d_{\textrm{area}} = 0.6$ &
$\SD d_{\textrm{angle}} = 4.3$ & $\SD d_{\textrm{area}} = 2.5$ &
$\SD d_{\textrm{angle}} = 16.3$ & $\SD d_{\textrm{area}} = 0.2$ \\
\end{tabular}
}
\caption{The result of surface reconstruction by geometry image of Human Brain triangular mesh with $100 \times 100$ geometry image, in which $d_{\textrm{angle}}$ and $d_{\textrm{angle}}$ are defined in \eqref{eq:angle_d} and \eqref{eq:area_d}, respectively.}
\label{fig:Brain Reconstruct}
\end{figure}

\subsection{Extreme Case of Reconstruction}
In the case of the Stanford Bunny, the angle-preserving property of distortion-balancing mapping fails surface reconstruction of geometry images. This occurs because the points in the planar parameterization corresponding to the Bunny's ear are tightly clustered to maintain the triangle's angles. As a result, accurately reconstructing Bunny's ear necessitates an extremely high number of points, which is impractical in a geometry image.

To address this issue, we enhance the area-preserving property in distortion-balancing mapping, which helps maintain uniform point distribution in the planar domain, by modifying the constraint as follows,
\begin{equation}
\min E_C(\bf) ~~\text{subject to}~~ \mu E_A(\bf) = E_C (\bf)
\label{eq:mu Ea}
\end{equation}
for $\mu > 0$.
By increasing $\mu$, we emphasize the area-preserving property over the angle-preserving aspect in the distortion-balancing parameterization. This adjustment improves the ability to accurately reconstruct the Bunny's ear using the geometry image (see Fig. \ref{fig:Standford Bunny mu}).

Lastly, we present the result of the Stanford Bunny's surface reconstruction using a $100 \times 100$ geometry image with distortion-balancing ($\mu = 15$ in \eqref{eq:mu Ea}), authalic parameterization and the original surface, as shown in Fig. \ref{fig:Standford Bunny comparison}. Both distortion-balancing and authalic mapping successfully reconstruct the detail of the surface.

\begin{figure}[]
\centering
\resizebox{\textwidth}{!}{
\begin{tabular}{cccc}
$\mu = 1$ & $\mu = 5$ & $\mu = 10$ & $\mu = 15$\\
\includegraphics[height=2.5cm]{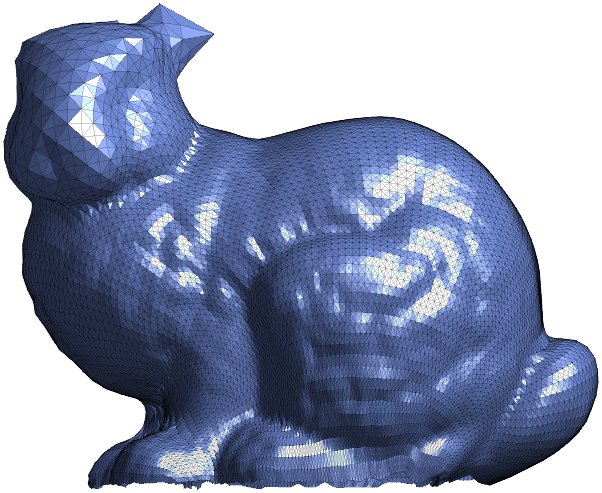} &
\includegraphics[height=2.8cm]{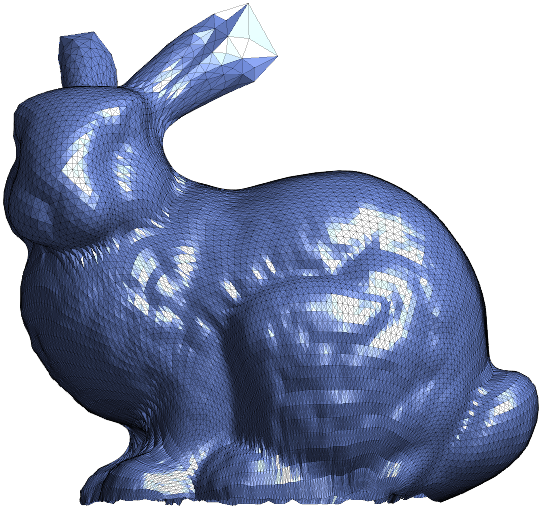} &
\includegraphics[height=3cm]{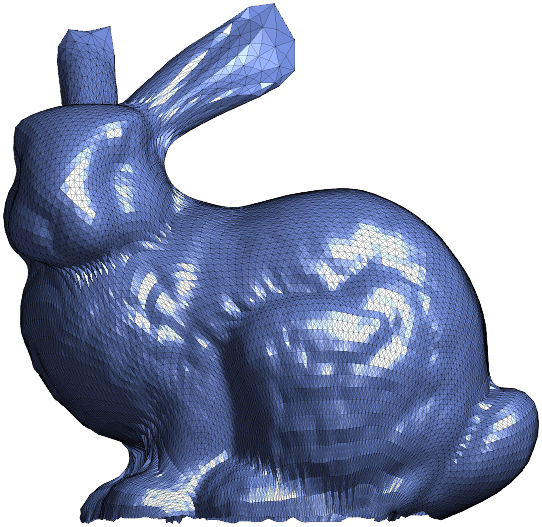} &
\includegraphics[height=3cm]{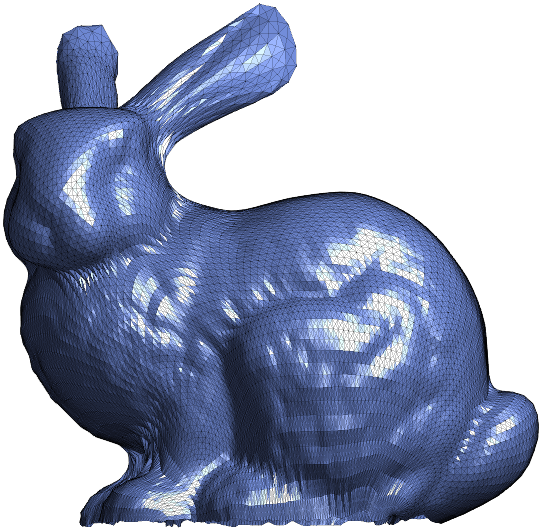} 
\end{tabular}
}
\caption{Surface reconstruction of the Stanford Bunny triangular mesh with the $100 \times 100$ geometry image by the minimization of the problem \eqref{eq:mu Ea} with different $\mu$.}
\label{fig:Standford Bunny mu}
\end{figure}

\begin{figure}[]
\centering
\resizebox{\textwidth}{!}{
\begin{tabular}{ccc}
Original & Distortion-Balancing & Authalic\\
\includegraphics[height=5cm]{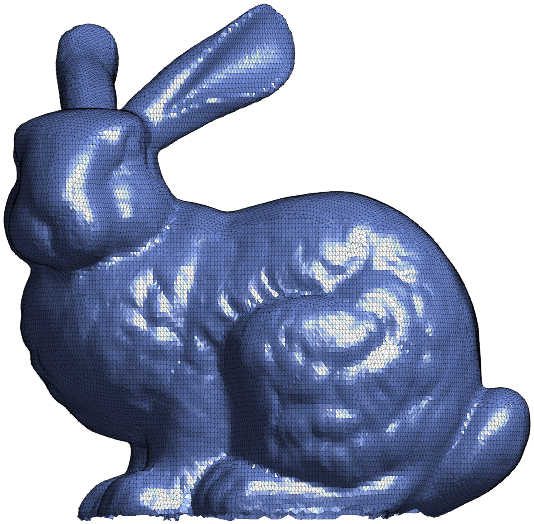} &
\includegraphics[height=5cm]{StanfordBunny_Remesh_BEM_15Ea.png} &
\includegraphics[height=5cm]{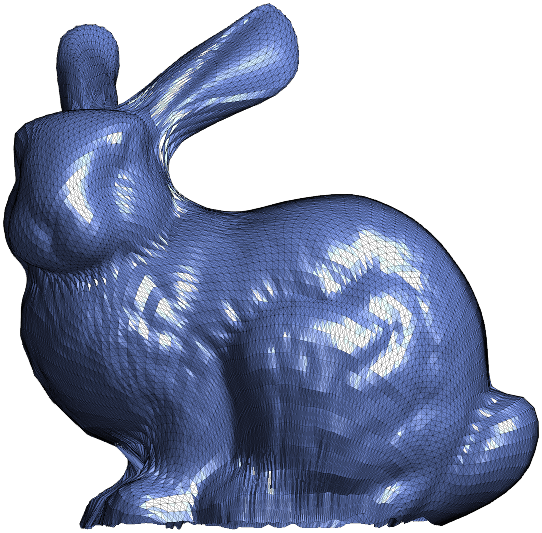} 
\end{tabular}
}
\caption{The original surface of the Stanford Bunny triangular mesh, reconstructed surface with $100 \times 100$ geometry image by distortion-balancing parameterization ($\mu = 15$ in \eqref{eq:mu Ea}), and reconstructed surface with $100 \times 100$ geometry image by authalic parameterization \cite{Yueh23}.}
\label{fig:Standford Bunny comparison}
\end{figure}

\section{Discussion}
\label{sec:7}
Different methods involving distortion-balancing parameterization employ various criteria to manage the trade-off between angular and area distortion. In this section, we compare our proposed method with other existing approaches and discuss the key differences in each method handling the trade-off.

In \cite{ChSh24}, the distortion-balancing mapping is obtained by forming a convex combination of the Beltrami coefficients from conformal and authalic mappings. The Beltrami coefficient, a complex-valued function, captures the local angular distortion of the mapping. However, this approach may not be ideal as it cannot capture the area distortion of mapping. Moreover, the combination coefficients are determined by minimizing the reconstruction error, which is the mean Hausdorff distance between the input and output meshes. This process can be numerically inefficient since it requires reconstruction at each iteration.

In \cite{NaSZ17}, the authalic mapping is obtained through optimal mass transportation (OMT) with an initial conformal mapping. The trade-off for balancing distortion is achieved by modifying the measure $\mu$ in the OMT process, in which the $\mu$ becomes a convex combination $\mu_t$ of the conformal image's area element and the original area element. Here, $\mu_0$ represents the area element induced by the conformal mapping, and $\mu_1$ represents the original area element. The suggested coefficient $t$ is 0.5, however, may be suboptimal. In addition, $\mu_t$ doesn't capture the angular distortion.

In \cite{Yueh23}, they minimized the balanced energy, which is formed by half of the conformal and authalic energy, as discussed in Sect. \ref{sec:5.2}. Although it captures both angular and area distortion, it may not be optimal in various triangular meshes (see Fig. \ref{fig:BEM_comparison}). Similarly, in \cite{YuHL20}, they form balance energy by convex combination of conformal and authalic energy with coefficient $\beta$ and the problem is formulated as $\max_\beta \min_f E_\beta(f)$, which may also not lead to well-balanced between angular and area distortion.

In contrast, our proposed method not only captures both angular and area distortion through conformal and authalic energy but also balances these distortions across various triangular meshes, offering greater flexibility and accuracy.

That said, the need for angle-preserving versus area-preserving properties can vary depending on the specific computational task. As discussed in Sect. \ref{sec:6}, area preservation may be more critical than angle preservation in applications like geometry images, particularly in extreme cases such as the Stanford Bunny triangular mesh (see Fig. \ref{fig:Standford Bunny mu}), where the authalic energy needs to be as low as one-fifteenth of the conformal energy.

\section{Concluding Remarks}
\label{sec:8}
We have proposed a new algorithm for computing disk-shaped parameterizations of simply connected open surfaces with equal conformal and authalic energy. This results in mappings that achieve an ideal trade-off between angular and area distortions. In addition, we have rigorously proven the global convergence of the algorithm to ensure that it consistently produces successful mappings.  Numerical experiments demonstrate that the algorithm produces bijective mappings with a well-balanced trade-off between angular and area accuracy across benchmark triangular meshes. Furthermore, we present the practical application of this balanced parameterization in representing surfaces as geometry images. In summary, the proposed formulation of distortion-balancing parameterization as a numerical optimization problem enables efficient and accurate computation of the mapping. The demonstrated effectiveness of the algorithm strongly supports its practical utility in real-world applications.

\appendix
\section{Initial Mapping by Fixed-Point Method}
In the previous work, Yueh et al. \cite{Yueh23} developed the fixed-point method to address the distortion-balancing parameterization. In our subproblem \eqref{eq:f update}, the fixed-point method is desirable for providing a good initial, who decrease energy rapidly initially. Numerically, we carry out 5 iterations tending to yield better efficiency.

Specifically, we drop the penalty term and recall the Lagrangian \eqref{eq:Lag} with fixed Lagrange multiplier $\lambda^{(k)}$ in $k$th iteration, denoted by
\[
\L_{\lambda^{(k)}} (\bf) \equiv \L(\bf, \lambda^{(k)}) =  E_C(\bf) + \lambda^{(k)} (E_A(\bf) - E_C(\bf)).
\]
Then, the gradient of $\L_{\lambda^{(k)}} (\bf)$ is
\[
\nabla_{\bf} \L_{\lambda^{(k)}} (\bf)  = (1 - \lambda^{(k)}) \nabla_{\bf}  E_C(\bf) + \lambda^{(k)} \nabla_{\bf}  E_A(\bf) \in \mathbb R^{n \times 2}. 
\]
To obtain the disk-shaped parameterization, the boundary points are selected by the arc length parameterization of $\partial \M$ and fixed during the iteration, denoted by $\bf_\B = \mathbbm b$. This implies $E_A$ and $E_C$ are equivalent to $E_S$ and $E_D$, respectively since the total area of the surface remains unchanged. Therefore, the gradient of $\L_{\lambda^{(k)}} (\bf)$ with respect to interior point $\bf_\I$ can be expressed by 
\[
\nabla_{\bf_\I} \L_{\lambda^{(k)}} (\bf)  = [L_{\lambda^{(k)}}(\bf)]_{\I, \I} \bf_\I + [L_{\lambda^{(k)}}(\bf)]_{\I, \B} \mathbbm b,
\] 
in which $\L_{\lambda^{(k)}}$ is the Laplacian, defined by
\begin{equation} \label{eq:L_B}
[L_{\lambda^{(k)}}(\bf)] = (1 - \lambda^{(k)}) L_D + 2\, \lambda^{(k)} L_S(\bf).
\end{equation}
The fixed-point method, inspired by the above gradient equal to zero, is the iterative method to solve 
\begin{equation} \label{eq:BEM}
    [L_{\lambda^{(k)}}(\bf^{(\ell)})]_{\I, \I} {\bf^s_\I}^{(\ell+1)} = - [L_{\lambda^{(k)}}(\bf^{(\ell)})]_{\I, \B} \mathbbm b^s
\end{equation} for $s = 1, 2$ in $\ell$th iteration, and $L_{\lambda^{(k)}}(\bf^{(0)}) = L_D$. The detail of the procedure is presented in Algorithm \ref{alg:BEM}

\begin{algorithm}[H]
\caption{Initial mapping with the fixed-point method}
\label{alg:BEM}
\begin{algorithmic}[1]
\Require A simply connected open triangular mesh $\M$, and a coefficient $\lambda \in [0,1]$. 
\Ensure A distortion-balancing initial mapping $\bf$.
\State Let $\I$ and $\B$ be defined as \eqref{eq:BI}.
\State Let $L = L_D$ as \eqref{eq:Ld}. 
\State Compute arc length boundary parameterization $\mathbbm b$.
\For{$k = 0, \cdots, 5$}. 
\State Solve linear systems $L_{\I,\I} \bf^s_{\I} = -L_{\I,\B} \mathbbm b^s$, for $s=1,2$.
\State Update $L$ by \eqref{eq:L_B}. 
\State $k \leftarrow k + 1$.
\EndFor
\end{algorithmic}
\end{algorithm}

\footnotesize
\bibliographystyle{abbrv}

\end{sloppypar}
\end{document}